\documentclass[leqno,10pt]{article}
\usepackage{amsmath}
\usepackage{amsfonts}
\usepackage{array}
\usepackage{enumerate}
\usepackage{graphicx,tikz}
\usepackage{amssymb}
\usepackage{amsthm}
\usepackage{xcolor}
\usepackage{chemarr}
\usepackage[margin=3cm]{geometry}
\usepackage{hyperref}
\usepackage[sort&compress,square,comma,numbers]{natbib}

\usepackage[version=3]{mhchem}

\newtheorem{theorem}{Theorem}
\newtheorem{proposition}[theorem]{Proposition}
\newtheorem{lemma}[theorem]{Lemma}
\newtheorem{corollary}[theorem]{Corollary}

\theoremstyle{definition}

\newtheorem{remark}{Remark}
\newtheorem{example}{Example}

\newcommand{\R}{\mathbb{R}}
\newcommand{\N}{\mathbb{N}}

\newcommand{\mZ}{\mathcal{Z}}
\newcommand{\mS}{\mathcal{S}}
\newcommand{\mR}{\mathcal{R}}
\newcommand{\mX}{\mathcal{X}}
\newcommand{\mN}{\mathcal{N}}
\newcommand{\mE}{\mathcal{E}}
\newcommand{\mG}{\mathcal{G}}
\newcommand{\mC}{\mathcal{C}}
\renewcommand{\k}{\kappa}

\newcommand{\norm}[2][\relax]{\ifx#1\relax \ensuremath{\left\Vert#2\right\Vert} \else \ensuremath{\left\Vert#2\right\Vert_{#1}}\fi}

\begin{document}

\title{Quasi-steady state and singular perturbation reduction\\ for  reaction networks with non-interacting species}

\author{Elisenda Feliu\\Department of Mathematical Sciences, University of Copenhagen\\Universitetsparken 5, 2100
Copenhagen, Denmark\\ \tt{efeliu@math.ku.dk}\\
\\
Christian Lax\\
Lehrstuhl A f\"ur Mathematik, RWTH Aachen\\
52056 Aachen, Germany\\
\tt{christian.lax@matha.rwth-aachen.de}\\
\\
Sebastian Walcher\\
Lehrstuhl A f\"ur Mathematik, RWTH Aachen\\
52056 Aachen, Germany\\
\tt{walcher@matha.rwth-aachen.de}\\
\\
Carsten Wiuf\\Department of Mathematical Sciences, University of Copenhagen\\Universitetsparken 5, 2100
Copenhagen, Denmark\\ \tt{wiuf@math.ku.dk}\\
}


\maketitle
\begin{abstract}Quasi-steady state (QSS) reduction is a commonly used method to lower the dimension of a differential equation model of a chemical  reaction network. From a mathematical perspective, QSS reduction is generally interpreted as a special type of singular perturbation reduction, but in many instances the correspondence is not worked out rigorously, and the QSS reduction may yield incorrect results. The present paper contains a thorough discussion of QSS reduction and its relation to singular perturbation reduction for the special, but important, case when the right hand side of the differential equation is linear in the variables to be eliminated. For this class we give necessary and sufficient conditions for a singular perturbation reduction (in the sense of Tikhonov and Fenichel) to exist, and to agree with QSS reduction. We then apply the general results to chemical reaction networks with non-interacting species, generalizing earlier results and methods for steady states to quasi-steady state scenarios. We provide easy-to-check graphical conditions to select parameter values yielding to singular perturbation reductions and additionally, we identify a choice of parameters for which the corresponding singular perturbation reduction agrees with the QSS reduction. Finally we consider a number of examples.

\end{abstract}

\section{Introduction}

Mathematical modelling of chemical reaction networks naturally yields a large class of (typically) polynomial ordinary differential equations (ODEs). Beyond chemistry and biochemistry, this class of differential equations is also useful in applications to ecology, epidemiology and genetics. These ODE systems often contain many variables (e.g.\ concentrations of chemical species) and parameters e.g.\ rate constants), and moreover the parameters may not be exactly identifiable by experiments. Therefore a general study of systems with undetermined parameters is appropriate, and possible reduction of such systems to smaller dimension is particularly relevant.

Quasi-steady state (QSS) reduction is commonly used  in (bio-)chemistry and related fields, but rarely for differential equations in other areas of application. The method was introduced, for a model of an enzyme-catalyzed reaction, by Michaelis and Menten in 1917 and (based on different assumptions)  by Briggs and Haldane \cite{briggshaldane} in 1925. The reasoning of these authors used intuition about the specifics and parameters of such reactions. In contrast, the insight that singular perturbation theory (based on Tikhonov's 1952 paper \cite{tikh}) can explain quasi-steady state phenomena is due to Heineken et al. \cite{hta} in 1967, and the analysis of the basic Michelis-Menten reduction was brought to a conclusion by Segel and Slemrod \cite{SSl} in 1989. Nowadays QSS is widely seen as  a special type of a singular perturbation scenario in the sense of Tikhonov and Fenichel \cite{fenichel}, although many authors still use QSS reduction without verifying the necessary conditions for singular perturbation reduction.  For the reader's convenience, we give a brief outline of singular perturbation reduction (including a coordinate-free version) in the Appendix.

A general mathematical study of quasi-steady state reduction, including consistency and validity requirements, and establishing agreement of QSS and singular perturbation reduction under rather restrictive conditions, was carried out in \cite{gwz3}. In the present paper we continue this study for a special (but quite relevant) class of differential equations. Building on   \cite{gw2,gwz3} we study ODE systems that depend linearly on the variables to be eliminated. To compute the singular perturbation reduction, we use the coordinate-free approach introduced in \cite{gw2} (see also Appendix 4 in \cite{gwz3}). We thus determine explicitly the two types of reduction for this general class of systems, and then obtain necessary and sufficient conditions for their agreement (up to higher order terms).

The motivation for studying this class of differential equations comes from reaction networks where the correct form of the ODE system can be identified by means of a set of \emph{non-interacting} species (under the assumption of mass-action kinetics), as introduced in \cite{feliu:intermediates,fwptm,saez_reduction}. A set of non-interacting species (variables) can be inferred from the reactions of the network alone without scrutinizing the analytical form of the ODE system, and will appear linearly in the differential equations.

In the setting of non-interacting species, we give general necessary and sufficient criteria for the existence of a Tikhonov-Fenichel reduction, and furthermore we provide necessary and sufficient criteria for the Tikhonov-Fenichel reduction to agree with the QSS reduction. The linear structure of the ODE system (in the variables to be eliminated)  provides easy to check sufficient \emph{graphical} criteria. For smaller reaction networks, the graph and the criteria can easily be constructed and checked by hand, providing criteria that are readily usable by application-oriented scientists.

We end the paper with a number of examples to illustrate the usefulness and limits  of the graphical approach.  In particular we study so-called post-translational modification (PTM) systems of which the classical Michaelis-Menten system is a special case, and further examples from the chemical and ecological literature.

\section{Linear elimination and reduction}

\subsection{Motivation and background: Reaction networks}\label{motivation}
We consider a system of ordinary differential equations (ODEs) arising from a   reaction network. A reaction network (or network for short) with species set $\mS=\{X_1,\ldots,X_n\}$ consists of a set of reactions $\mR=\{r_1,\ldots,r_m\}$, such that the $i$-th reaction takes the form
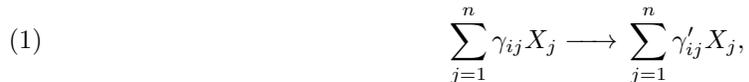
\begin{equation}\label{eq:network}
 \sum_{j=1}^n \gamma_{ij} X_j \ce{->} \sum_{j=1}^n \gamma'_{ij} X_j,
\end{equation}
where $\gamma_{ij},\gamma'_{ij}\in\N_0$, are non-negative integers. The left hand side is called the \textit{reactant} of the reaction, the right hand side, the \textit{product}, and jointly they are \textit{complexes}. We will assume that any species takes part in at least one reaction and that the reactant and product sides are never identical, that is, it cannot be that $\gamma_{ij}=\gamma'_{ij}$ for all $j=1,\ldots,n$.

A reaction network gives rise to an ODE system of the form
\begin{equation}\label{eq:crneq}
\dot{y} =N v(y), \qquad y\in \R^n_{\geq 0},
\end{equation}
where $\dot{y}$ denotes derivative with respect to time $t$, $N$ is the stoichiometric matrix, that is, the $i$-th column of $N$ is the vector with $j$-th entry $\gamma'_{ij}-\gamma_{ij}$, and $v(y)$ is a vector of \emph{rate functions} defined on an open neighborhood of $\R^m_{\geq 0}$ and  {\em non-negative} on $\R^m_{\geq 0}$.  
 Under the additional assumption that $v_i(y)$ vanishes whenever $y_j=0$ and $\gamma'_{ij}<\gamma_{ij}$, that is, whenever $X_j$ is consumed by the reaction,  the non-negative orthant $\R^n_{\geq 0}$ as well as the positive orthant $\R^n_{>0}$ are forward invariant by the trajectories of the system \cite{Sontag:2001}.   Of particular interest is mass-action kinetics  with
$$v=(v_1,\ldots,v_m), \quad v_i(y)=\kappa_i \prod_{j=1}^n y_j^{\gamma_{ij}},\quad  \kappa_i> 0,$$
where $\kappa_i$ is the (non-negative) reaction rate constant. Formally, the borderline case $\kappa_i=0$ corresponds to removing a reaction, and we are also interested in such scenarios. 

In \cite{feliu:intermediates,fwptm,saez_reduction}, a reduction procedure was introduced for the computation and discussion of  steady states of an ODE system \eqref{eq:crneq}. It centers around the algebraic elimination at steady state of  variables representing the concentrations of so-called \textit{non-interacting species}:  Let $\mZ=\{Z_1,\ldots,Z_P\}\subseteq \mS$ be a subset of the species set and let $\mX=\mS\setminus \mZ=\{X_1,\ldots,X_n\}$ be the complementary subset (where the species potentially are relabelled compared to \eqref{eq:network}). If, after writing the reactions as
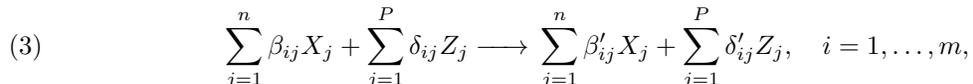
\begin{equation}\label{eq:reactions}
 \sum_{j=1}^n \beta_{ij} X_j+\sum_{j=1}^P \delta_{ij} Z_{j}\ce{->} \sum_{j=1}^n \beta'_{ij} X_j+\sum_{j=1}^P \delta'_{ij} Z_j, \quad  i=1,\ldots,m, 
\end{equation}
the conditions
$$\sum_{j=1}^P \delta_{ij}\le 1,\quad \sum_{j=1}^P \delta'_{ij}\le 1$$
are satisfied, then the set $\mZ$ is said to be  \emph{non-interacting} and its elements are called \emph{non-interacting species}. Intuitively, two or more  species are non-interacting if they are never found together in the same complex.  This implies for mass-action kinetics that the variables corresponding to non-interacting species appear linearly in the ODE system. 
The vector of concentrations of the non-interacting species and the remaining species are denoted by $z=(z_1,\ldots,z_P)$ and $x=(x_1,\ldots,x_n)$, respectively. Hence in this terminology $y=(x,z)\in \R^{n+P}$.  It was shown in  \cite{fwptm} that  one may parameterize $z_1,\ldots, z_P$ by  $x$ at steady state, given suitable regularity conditions on the Jacobian matrix and assuming the rate functions are linear in $z$. In the present paper we will extend this reduction procedure to the case of quasi-steady state reduction.

\subsection{{The general setting}}\label{gensetting}

We first discuss the reduction procedures in a context that is not restricted to  reaction networks, just keeping the characteristic property of reaction equations with non-interacting species. This class of equations, and their reductions, may be of interest beyond  reaction networks. 

We consider a parameter dependent ODE system that is linear in $z$, that is,
\begin{equation}\label{generalsystem}
\begin{array}{rcccl}
\dot x&=&a(x,\pi)&+&A(x,\pi)z\\
\dot z&=&b(x,\pi)&+&B(x,\pi)z,
\end{array}
\end{equation}
with the following conditions on the domains of definition and the functions: 
\begin{itemize}
\item $\pi$ is a parameter vector varying in a subset $\Pi$ of $\R^q$ (for some $q$), with the property that for every $\widetilde\pi\in \Pi$ there is a smooth curve $\sigma\colon [0,\,1]\to \Pi$ such that $\sigma(0)=\widetilde \pi$ and $\sigma(s)\in{\rm int}\,\Pi$ for all $s>0$; in particular ${\rm int}\,\Pi$ is dense in $\Pi$. (For applications to reaction networks, we will have $\mathbb R^q_{> 0}\subseteq \Pi\subseteq\mathbb R^q_{\geq 0}$.)
\item There are distinguished open and non-empty sets $V_1\subseteq \mathbb R^n$, $V_2\subseteq \mathbb R^p$ such that 
\begin{itemize}
\item  $a,\,b,\,A,\,B$ are (vector-valued resp.\ matrix-valued) functions of dimension $n\times 1$,  $p\times 1$, $n\times p$, and $p\times p$, respectively, that are defined and sufficiently differentiable on an open neighborhood of $\overline{V_1}\times\overline{\Pi}$;
\item $\overline{V_1}\times\overline{V_2}$ is positively invariant for system \eqref{generalsystem}, for any $\pi\in\Pi$.
\end{itemize}
(In applications to reaction networks, we will have $V_1=\mathbb R^n_{>0}$ and $V_2=\mathbb R^p_{>0}$,  where $p\le P$ is to be defined later.)
\item For given $\pi\in\Pi$ define
\[
\Omega_\pi^*:=\big\{x\,|\,a(x,\pi),\,A(x,\pi),\,b(x,\pi),\,B(x,\pi),\text{ are defined and sufficiently differentiable}\big\}
\]
and
\[
\Omega_\pi:=\big\{x\in\Omega_\pi^*\,|\,B(x,\pi) \text{ is invertible}\big\},
\]
noting that both are open subsets of $\mathbb R^n$  and $V_1\subseteq \Omega^*_{\hat\pi}$. 
\end{itemize}

The general question is whether (and how) it is possible  to eliminate $z$, thus obtaining a reduced system in $x$ alone. For  reaction networks, the ``classical'' quasi-steady state (QSS) reduction (see Briggs and Haldane \cite{briggshaldane}, Segel and Slemrod \cite{SSl} for the Michaelis-Menten system, and many others) has been in use for a long time. 
For this heuristic reduction procedure one assumes invertibility of $B(x,\pi)$ for all $x$ and furthermore assumes that the rate of change for $z$ is equal to zero (or rather, almost zero in a relevant time regime). Then the ensuing algebraic relation
\[
0=\dot z= b(x,\pi)+B(x,\pi)z
\]
yields the QSS reduced system
\begin{equation}\label{cqssredsys}
\dot x=a(x,\pi)-A(x,\pi)B(x,\pi)^{-1} b(x,\pi),\qquad {x\in \Omega_\pi.}
\end{equation}

A priori this is a formal procedure, and one should not generally expect any similarity between solutions of \eqref{generalsystem} and \eqref{cqssredsys}. But this may be the case in certain parameter regions. A general discussion of QSS reductions, consistency conditions and their relation to singular perturbations was given in \cite{gwz3}. In the present paper we will obtain detailed results for systems of the special type \eqref{generalsystem}.

When Michaelis and Menten, and Briggs and Haldane, introduced QSS reduction, singular perturbation theory did not even exist. But starting with the seminal paper \cite{hta} by Heineken et al., the interpretation of QSS reduction as a singular perturbation reduction in the sense of Tikhonov \cite{tikh} and Fenichel \cite{fenichel} has been established in the literature. A convenient version of the classical reduction theorem is stated in Verhulst \cite{verhulst}, Thm.~8.1. We will refer to this version, with all differential equations autonomous.

We need to adjust system \eqref{generalsystem} by introducing a ``small parameter''  \cite{gw2,gwz}. To this end, we fix a suitable (to be specified below) parameter value $\widehat \pi$, consider a curve $\varepsilon\mapsto \widehat\pi+\varepsilon\pi^*+\cdots \in \Pi$ in the parameter space and expand
\[
\begin{array}{rcl}
a(x,\widehat \pi +\varepsilon\pi^*+\cdots)&=&a_0(x)+\varepsilon a_1(x)+\cdots\\
A(x,\widehat \pi +\varepsilon\pi^*+\cdots)&=&A_0(x)+\varepsilon A_1(x)+\cdots\\
b(x,\widehat \pi +\varepsilon\pi^*+\cdots)&=&b_0(x)+\varepsilon b_1(x)+\cdots\\
B(x,\widehat \pi +\varepsilon\pi^*+\cdots)&=&B_0(x)+\varepsilon B_1(x)+\cdots,
\end{array}
\]
where $a(x,\widehat\pi)=a_0(x)$,  $A(x,\widehat\pi)=A_0(x)$, etc., to obtain a system with small parameter $\varepsilon$, which we rewrite in the form
\begin{equation}\label{epsilonsystem}
\begin{array}{rcccccl}
\dot x&=&a_0(x)+A_0(x)z&+&\varepsilon\left(a_1(x)+A_1(x)z\right)&+&\cdots\\
\dot z&=&b_0(x)+B_0(x)z&+&\varepsilon\left(b_1(x)+B_1(x)z\right)&+&\cdots
\end{array}
\end{equation}
for {$x\in \Omega_{\widehat\pi}$ and $z\in V_2$. }
As before, for $a_0(x),b_0(x),$ $A_0(x), B_0(x)$, the dimensions are  $n\times 1$,  $p\times 1$, $n\times p$, and $p\times p$, respectively.

This system is in general not in the standard form for singular perturbations given in \cite{verhulst}, thus slow and fast variables are not separated. But assuming the existence of a transformation to standard form for singular perturbations, at $\varepsilon=0$ one has a positive dimensional local manifold of stationary points, usually called the {\em critical manifold}. In turn, the existence of such a critical manifold imposes conditions on the parameter value $\widehat\pi$. (For singular perturbations, this is part of the suitability mentioned above; see \cite{gwz,gwz3} for details.)

 The connection between QSS and singular perturbation reductions was generally discussed in \cite{gwz3}, Section 4. The special type of system \eqref{generalsystem} allows for a simplified, shorter discussion, as follows. 
\begin{itemize}
\item We consider throughout $\pi$ such that $\Omega_\pi\not=\emptyset$, hence $B(x,\pi)$ is invertible for some $x$.
\item As shown in \cite{gwz3}, Proposition 2, the minimal requirement for QSS to be consistent for all small perturbations of $\widehat \pi$ is invariance of  the {\em QSS variety}
\[
Y_{\widehat\pi}:=\big\{(x,z)\in\Omega_{\widehat\pi}\times{V_2}\,|\,B(x,\widehat\pi)z+ b(x,\widehat\pi)=0\big\}.
\]
This requirement guarantees for small perturbations of $\widehat\pi$ that the $x$--components of solutions to \eqref{generalsystem} with initial value in $Y_{\widehat\pi}$ remain close to the corresponding solutions of \eqref{cqssredsys}, and it is also necessary for this property. 
\item The invariance condition may be expressed as
\[
\left(Db_0(x)-DB_0(x)(B_0(x)^{-1}b_0(x))\right)\left(a_0(x)-A_0(x)B_0(x)^{-1} b_0(x)\right)=0,\quad {x\in \Omega_{\widehat\pi}.}
\]
(See \cite{gwz3} for the general form; to verify directly in the given setting, evaluate $\frac{d}{dt}\left(B_0(x)z+ b_0(x)\right)=0$ on $Y_{\widehat\pi}$.)
\item The invariance condition alone is too weak to ensure quasi-steady state properties on par with expectations concerning fast-slow timescales. As a simple example, consider the case $b_0=0$, with system
\[
\begin{array}{rcccccl}
\dot x&=&a_0(x)+A_0(x)z&+&\varepsilon\left(a_1(x)+A_1(x)z\right)&+&\cdots\\
\dot z&=&B_0(x)z&+&\varepsilon\left(b_1(x)+B_1(x)z\right)&+&\cdots,\\
\end{array}
\]
and the QSS variety given by $z=0$. Whenever $a_0$ is non-zero, the rate of change for $x$ is of order one on the QSS manifold characterized by ($\dot z=0$, hence) $z=O(\varepsilon)$. 

\item Conclusion (see also the extended discussion in \cite{gwz3}, Section 4): The natural way of transferring the QSS assumption for $z$ to a singular perturbation scenario is to stipulate that at $\varepsilon=0$, the equation $B_0(x)z+b_0(x)=0$ defines a set of stationary points of system  \eqref{epsilonsystem}.
\end{itemize}

 The above definitions and reasoning lead us to assume the following conditions in the sequel.\\
 
\noindent{\bf Blanket conditions.}
 \begin{enumerate}[(i)]
\item $\Omega_{\widehat\pi}\cap V_1\not=\emptyset$.
\item $z=-B_0(x)^{-1}b_0(x)$ defines a set of stationary points of system \eqref{epsilonsystem} for $\epsilon=0$, thus 
\[
a_0(x)-A_0(x)B_0(x)^{-1}b_0(x)=0
\]
holds for all {$x$ in $\Omega_{\widehat\pi}$}. We denote the critical set defined by $z=-B_0(x)^{-1}b_0(x)$ by $Y_{\widehat\pi}$.
\end{enumerate}

 One may rephrase the second condition for the original parameter dependent system \eqref{generalsystem} as an identity
\begin{equation*}\label{tfpvcond}
a(x,\widehat\pi)-A(x,\widehat\pi)B(x,\widehat\pi)^{-1} b(x,\widehat\pi)=0 \quad\text{  for all  }  {x\in \Omega_{\widehat\pi},}
\end{equation*}
which, in turn, imposes conditions on the parameter value $\widehat \pi$. Thus we obtain a special instance of a {\em Tikhonov-Fenichel parameter value, briefly TFPV}, as introduced in \cite{gwz}, but for a prescribed critical manifold.

We now turn to reductions of system \eqref{epsilonsystem}, starting with the singular perturbation reduction with prescribed critical manifold $Y_{\widehat\pi}$. It will be convenient to introduce
\[
w(x):=B_0(x)^{-1}b_0(x)\quad  \in \R^p\textrm{ for }{x\in \Omega_{\widehat\pi},}
\]
and thus have
\[
 z=-w(x) \text{  on  }Y_{\widehat\pi}
\]
by the blanket conditions, as well as
\[
\begin{array}{rcl}
b_0(x)+B_0(x)z&=&B_0(x)\left(w(x)+z\right),\\
a_0(x)+A_0(x)z&=&A_0(x)\left(w(x)+z\right).
\end{array}
\]

We now carry out the decomposition and reduction procedure from \cite[Theorem 1, Remarks 1 and 2]{gw2}  with
\begin{equation}\label{h0eq}
h^{(0)}(x,z)=\begin{pmatrix}a_0(x)+A_0(x)z\\
                               b_0(x)+B_0(x)z\end{pmatrix}, \quad h^{(1)}(x,z)=\begin{pmatrix}a_1(x)+A_1(x)z\\
                               b_1(x)+B_1(x)z\end{pmatrix}.
\end{equation}
(We note that, in the given situation, reduction formulas provided earlier by Fenichel \cite{fenichel} and Stiefenhofer \cite{sti} are also applicable.)
 The following is a straightforward application of \cite{gw2}.

\begin{lemma}\label{lem:lessuglylemma}
Assume blanket conditions (i) and (ii) hold.

\begin{enumerate}[(a)]
\item \emph{Decomposition}: One has
\[
\begin{pmatrix}a_0(x)+A_0(x) z\\
                               b_0(x)+B_0(x) z\end{pmatrix}=\begin{pmatrix}A_0(x)\\
                               B_0(x)\end{pmatrix}\cdot (w(x)+z)=P(x)\cdot \mu(x,z),
\]
where $P(x)=\begin{pmatrix}A_0(x)\\
                               B_0(x)\end{pmatrix}$ and $\mu(x,z)=w(x)+z$ is a map from {$\Omega_{\widehat\pi}\times \R^p$} to $\R^p$, and furthermore
\[
D\mu(x,z)=\Big(Dw(x)\quad I_p \Big) \ \in \R^{p\times (n+p)}.
\]
(Here $I_p$ denotes the $p\times p$ identity matrix.)
\item \emph{Condition for reducibility}: Define
\[
\begin{array}{rcl}
\Delta(x)  &:=&D\mu(x,z) \cdot P(x) \\
         &=&Dw(x)\,A_0(x)+B_0(x)=M(x) B_0(x)\quad \text{on }Y_{\widehat\pi},
\end{array}
\]
with
\[
M(x):=Dw(x) A_0(x) B_0(x)^{-1}+I_p \quad  \in \R^{p\times p}\quad \text{on }Y_{\widehat\pi}.
\]
Then a (local) Tikhonov-Fenichel reduction with a linearly attractive critical manifold $Y_{\widehat\pi}$ exists if and only if the open set
\[
\widetilde\Omega_{\widehat\pi}:=\left\{x\in \Omega_{\widehat\pi}\,|\,\text{ all eigenvalues of } \Delta(x)\text{ lie in the open left half plane }\right\}\subseteq \Omega_{\widehat\pi}
\]
is non-empty.
\item Under the conditions stated in (b), the reduced system on $Y_{\widehat\pi}$, in slow time scale $\tau=\varepsilon t$, is obtained by multiplication of the projection matrix
\begin{align*}
Q(x) &:=I_{n+p}-\begin{pmatrix}A_0(x)\Delta(x)^{-1}\\B_0(x)\Delta(x)^{-1}\end{pmatrix}\cdot \Big(Dw(x)  \quad I_p\Big)\\
 &=I_{n+p}-\begin{pmatrix}A_0(x)B_0(x)^{-1}M(x)^{-1}\\ M(x)^{-1}\end{pmatrix}\cdot\Big(Dw(x) \quad I_p\Big)\\
\end{align*}
with 
\[
h^{(1)}(x)=\begin{pmatrix}a_1(x)-A_1(x)w(x)\\ b_1(x)-B_1(x)w(x)\end{pmatrix}
\]
as in \eqref{h0eq}.
\end{enumerate}
\end{lemma}

To keep notation manageable, we now suppress the argument $x$ in $a_0,a_1,\, A_0,A_1,\,b_0,b_1\,B_0,B_1\,w, \Delta, M$ and their derivatives. 
The essential part of the reduction is given in the following proposition. Note that on $Y_{\widehat\pi}$ it suffices to consider the equation for $x$.

\begin{proposition}
 Given the blanket conditions (i) and (ii), assume that the conditions in part (b) of  Lemma \ref{lem:lessuglylemma} hold, and consider $Q$ as a $2\times 2$ block matrix of dimension $(n,p)\times(n,p)$. Then the upper left block equals
\[
I_n-A_0\,B_0^{-1}\,M^{-1}\,Dw
\]
and the upper right block equals
\[
-A_0\,B_0^{-1}\,M^{-1}.
\]
The reduced equation, in slow time $\tau=\varepsilon t$, on $Y_{\widehat\pi}$ yields the system
\begin{equation}\label{TFred}
\frac{dx}{d\tau}=\left(I_n-A_0\,B_0^{-1}\,M^{-1}\,Dw\right)\left(a_1-A_1w\right)-\left(A_0\,B_0^{-1}\,M^{-1}\right)\left(b_1-B_1w\right)
\end{equation}
for the projection of a solution $(x(\tau),\,z(\tau))$ of the reduced system on $Y_{\widehat\pi}$ to its first component. This may be rewritten as
\[
\frac{dx}{d\tau}=
\left(I_n-A_0\,(DwA_0+B_0)^{-1}\,Dw\right)\left(a_1-A_1w\right)-A_0\,(DwA_0+B_0)^{-1}\left(b_1-B_1w\right).
\]
\end{proposition}

This general reduction formula may seem rather unwieldy, given the seemingly simple starting point \eqref{generalsystem}. For the purpose of illustration we look at the smallest dimension.

\begin{proposition}
For $n=p=1$ the reduced system in slow time is given by
\begin{equation*}
\frac{dx}{d\tau}=\frac{(B_0a_1-A_0b_1)-(B_0A_1-A_0B_1)w}{B_0+w^\prime A_0}.
\end{equation*}
\end{proposition}

We also make note of an important special case.

\begin{proposition}\label{zeezeroprop}
When $w$ is constant, then the reduced equation is given by 
\begin{equation*}
\frac{dx}{d\tau}=a_1-A_1w-A_0B_0^{-1}(b_1-B_1w).
\end{equation*}
In particular when $w=0$ (thus the critical manifold is given by $z=0$) the reduced equation in slow time reads
\begin{equation*}
\frac{dx}{d\tau}=a_1-A_0B_0^{-1}b_1.
\end{equation*}
\end{proposition}

Having obtained the singular perturbation reduction, we compare it to the classical quasi-steady state reduction.
\begin{proposition}\label{qssprop} Assume blanket conditions (i) and (ii) hold. 
\begin{enumerate}[(a)]
\item The classical quasi-steady state reduction of system \eqref{epsilonsystem} yields the QSS-reduced system 
\begin{equation}\label{qssredsysslow}
\frac{d x}{d\tau}=\left(a_1-A_1w-A_0B_0^{-1}(b_1-B_1w)\right)+\varepsilon(\cdots)
\end{equation}
in slow time.
\item The classical QSS reduction agrees with the singular perturbation reduction (up to higher order terms in $\varepsilon$) if and only if 
\begin{equation}\label{agreecond}
A_0B_0^{-1}M^{-1}\,Dw\big(A_0B_0^{-1}\left(B_1w-b_1\right)-\left(A_1w-a_1\right)\big)=0.
\end{equation}
Given this condition, Tikhonov's theorem also applies to the QSS reduction.
\end{enumerate}
\end{proposition}

\begin{proof}The second equation in \eqref{epsilonsystem} shows
\[
z=-(B_0+\varepsilon B_1+\cdots)^{-1}(b_0+\varepsilon b_1+\cdots).
\]
With the geometric series one has
\[
\begin{array}{rcl}
(B_0+\varepsilon B_1+\cdots)^{-1}&=&B_0^{-1}(I+\varepsilon B_1B_0^{-1}+\cdots)^{-1}\\
                                                  &=& B_0^{-1}-\varepsilon B_0^{-1}B_1B_0^{-1}+\cdots
\end{array}
\]
and therefore
\[
z=- B_0^{-1}b_0+\varepsilon(B_0^{-1}B_1B_0^{-1}b_0-B_0^{-1}b_1)+\cdots
\]
Substitution into the first equation of \eqref{epsilonsystem}, replacing $b_0=B_0w$ and using blanket condition (ii), and further collecting terms yields the assertion of part (a). As for part (b), comparing equations \eqref{TFred} and \eqref{qssredsysslow} one obtains as necessary and sufficient conditions:
\[
\begin{array}{crcl}
 & A_0B_0^{-1}\left(B_1w-b_1\right)&=& A_0B_0^{-1}M^{-1}\left(B_1w-b_1\right)\\
& &+& A_0B_0^{-1}M^{-1}\,Dw\left(A_1w-a_1\right)\\
 \Leftrightarrow& A_0B_0^{-1}\left(I_p-M^{-1}\right)\left(B_1w-b_1\right)&=&A_0B_0^{-1}M^{-1}\,Dw\left(A_1w-a_1\right)\\
 \Leftrightarrow& A_0B_0^{-1}M^{-1}\left(M-I_p\right)\left(B_1w-b_1\right)&=&A_0B_0^{-1}M^{-1}\,Dw\left(A_1w-a_1\right)\\
\Leftrightarrow& A_0B_0^{-1}M^{-1}\,Dw\,A_0B_0^{-1}\left(B_1w-b_1\right)&=&A_0B_0^{-1}M^{-1}\,Dw\left(A_1w-a_1\right)\\
\end{array}
\]
recalling the definition of $M$ in the last step. The last assertion holds since (as noted in \cite{gwz3}) higher order terms in $\varepsilon$ are irrelevant for the convergence statement in Tikhonov's theorem.
\end{proof}

We recover a special case of \cite{gwz3}, Prop. 5.

\begin{corollary}\label{qsscor}When $w$ is constant (in particular when $b_0=0$), then the differential equations for the singular perturbation reduction and the QSS reduction in slow time agree up to terms of order $\varepsilon$. 
\end{corollary}

In general the QSS heuristic and singular perturbation reduction yield substantially different results, and the reduction by QSS is incorrect. However, for the following notable cases the two reductions are in agreement.
\begin{itemize}
\item $Dw=0$, thus $w$ is constant, see Corollary \ref{qsscor}. (As will turn out, this case occurs for many reaction networks.)
\item $A_0=0$. Here, system \eqref{h0eq} is in Tikhonov standard form with slow and fast variables separated.
\item $A_0B_0^{-1}\left(B_1w-b_1\right)=A_1w-a_1$. Here, both reductions have right hand side zero.
\end{itemize}
In dimension $2$ this list is complete (as seen by inspection of \eqref{agreecond}).

\medskip
To summarize: Given the blanket conditions (i) and (ii), as well as invertibility of the matrix $M$ (or $\Delta$) and the
 eigenvalue condition in Lemma \ref{lem:lessuglylemma}(b), we have determined a closed-form version of the reduced system by singular perturbations and clarified its relation to the classical QSS heuristic. The next step will be to apply these results to reaction networks, making use of their special properties.

\section{Application to  reaction networks}\label{crnsec}

We now return to the reaction networks from Section \ref{motivation}, and will extend results from \cite{feliu:intermediates,Fel_elim} on steady states to quasi-steady states. We assume that all the conditions stated after equation \eqref{generalsystem} hold, with $\mathbb R^q_{>0} \subseteq\Pi\subseteq \mathbb R^q_{\geq 0}$,
and  introduce the following further assumptions and definitions. Note that all requirements are satisfied for systems with mass-action kinetics.
\begin{itemize}
\item The species are ordered as $X_1,\dots,X_n,Z_1,\dots,Z_P$, such that the concentration vector is $(x,z)$.
\item The  rate function of a reaction  involving the non-interacting species $Z_i$ in the reactant is linear in $z_i$ and does not depend on any other $z_j$.
\item The rate function of a reaction that does not involve any non-interacting species in the reactant is constant in $z_i$, $i=1,\ldots,P$.
\item We order the set of reactions such that the first $m_1$ reactions only have species in $\mX$ in the reactant (without restrictions on  the product), and the last $m_2$ reactions  all have one non-interacting species in the reactant (and at most one in the product).  Thus,  $m=m_1+m_2$.
\end{itemize}

At the outset we consider a general parameter vector $\kappa=(\kappa_1,\ldots,\kappa_q)\in\Pi$ varying in the parameter set $\Pi$.   Let $v_1(x,\kappa)$ denote the vector of   rate functions for the first $m_1$ reactions (which by assumption do not depend on $z$) and let $v_2(x,z,\kappa)$ be the  vector of rate functions for the last $m_2$ reactions (which by assumption each component is linear in the concentration of the only non-interacting species in the reactant). Recall that, by assumption, all these functions are defined for all $(x,\kappa)$ in an open neighborhood of $\mathbb R^n_{\geq 0}\times \mathbb R^q_{\geq 0}$, and all $z\in \mathbb R^P$. We further assume  forward invariance of $\mathbb R^n_{\geq 0}\times \mathbb R^P_{\geq 0}$ for system \eqref{eq:crneq}.

The dynamical system \eqref{eq:crneq}, which evolves in $\R^{n+P}_{\ge 0}$ (by the invariance), may then be written as
\begin{align}\label{eq:sys1}
\begin{pmatrix} \dot{x} \\ \dot{z} \end{pmatrix} &=N v(x,z,\kappa)=
\begin{pmatrix} N_{11} & N_{12} \\  N_{21} & N_{22} \end{pmatrix} \begin{pmatrix} v_1(x,\kappa) \\ v_2(x,z,\kappa) \end{pmatrix} \nonumber \\
 &= \begin{pmatrix} N_{11}  \\ N_{21}  \end{pmatrix} v_1(x,\kappa) +\begin{pmatrix} N_{12}  \\ N_{22}  \end{pmatrix} v_2(x,z,\kappa).
\end{align}
Here the size of $N_{11}$ is $n\times m_1$, that of $N_{21}$ is $P\times m_1$, that of $N_{12}$ is $n\times m_2$ and that of $N_{22}$ is $P\times m_2$.
By assumption $v_2(x,z,\kappa)$ is linear in $z$, hence
\begin{equation}\label{eq:v2}
\begin{pmatrix}  N_{12} \\  N_{22} \end{pmatrix}  v_2(x,z,\kappa)= \begin{pmatrix}  K_1(x,\kappa) \\  K_2(x,\kappa) \end{pmatrix}  z,
\end{equation}
where $K_1$ is a $n\times P$ matrix and $K_2$ is  a $P\times P$ matrix.  We might take the system \eqref{eq:sys1} with \eqref{eq:v2} inserted to be of the form \eqref{generalsystem}. However, we will refrain from doing so here. Indeed we will modify the system before making the identification with \eqref{generalsystem}.

We note some crucial properties of $K_2(x,\kappa)$, most of which were already shown in \cite{feliu:intermediates,Fel_elim}.

\begin{lemma}\label{lem:compartment}
Let $\kappa\in\Pi$, and let $\Omega_\kappa$ be non-empty. Then the following hold.
\begin{enumerate}[(a)]
\item For all $x\in\Omega_\kappa\cap\mathbb R^n_{\geq 0}$,  $K_2(x,\kappa)$ is a compartmental matrix, that is, the diagonal entries of $K_2$ are non-positive, the off-diagonal entries non-negative, and all column sums are non-positive.
\item  For all $x\in\Omega_\kappa\cap\mathbb R^n_{\geq 0}$, all non-zero eigenvalues of $K_2(x,\kappa)$ have negative real part, and for the eigenvalue $0$ (if it occurs) the  geometric and algebraic multiplicity are equal.
\item Assume the rank of $K_2(x,\kappa)$ is equal to $p=P-k<P$ for all $x$ in an open set $\widetilde\Omega_\kappa\subseteq\Omega_\kappa$, and assume there are linearly independent linear forms $\lambda_1,\ldots,\lambda_k$ on $\mathbb R^P$ such that $\lambda_i (K_2(x,\kappa))=0$ for all $x$ and $\kappa$, and $1\leq i\leq k$. Then $K_2(x,\kappa)$ restricts to a linear map on ${\rm Ker}\,\lambda_1\cap\cdots\cap {\rm Ker}\,\lambda_k$. This map is invertible, and its eigenvalues are just the non-zero eigenvalues of  $K_2(x,\kappa)$.
\end{enumerate}
\end{lemma}

\begin{proof} (a) A column of $N_{22}$ contains one entry $-1$ and one entry $1$ (with all other entries zero) if it corresponds to a reaction in \eqref{eq:reactions} with some non-interacting species appearing on either side, and contains just one entry $-1$ (with all other entries equal to zero) if it corresponds to a reaction with a non-interacting species appearing just on the left hand side. Since the rate functions are non-negative, the assertion of (a) follows.

(b) The first assertion of (b) is well known, see e.g.  Anderson \cite{anderson-compartmental}, Thm.~12.1 or Chapter 6 of Berman and Plemmons \cite{berman} (noting that compartmental matrices are negative M-matrices). We include a proof of the second statement (which also is known) for the sake of completeness: Abbreviate $F:=K_2(x,\kappa)$, with $(x,\kappa)$ fixed, and consider the linear differential equation $\dot z=F\cdot z$. For this equation the positive orthant is positively invariant, and the equation admits the Lyapunov function $\sum_{i=1}^P z_i$, whence all solutions in the positive orthant are bounded for positive times. The existence of a non-trivial Jordan block for the eigenvalue $0$ would imply the existence of unbounded solutions for positive times; a contradiction.

(c) We have shown in (b) that $\mathbb R^P$ is the direct sum of the kernel and the image of $F$. Since the image  is contained in ${\rm Ker}\,\lambda_1\cap\cdots\cap {\rm Ker}\,\lambda_k$, and both have dimension $P-k$, they are equal. This shows invertibility and the assertion about the eigenvalues, since ${\rm Im}\,F$ is the 
sum of generalized eigenspaces for non-zero eigenvalues.
\end{proof}

The following remarks further illustrate the general structure of $K_2(x,\kappa)$.
\begin{itemize}
\item For fixed $\kappa$ one may always consider those $x$ for which $K_2(x,\kappa)$ has maximal rank, but this may force restriction to a non-empty open subset of $\Omega_\kappa$ (with some consequences for applying the reduction results from Subsection~\ref{gensetting}). In the case of mass-action kinetics this subset is open and dense.
\item The  relevant case of irreducible $K_2(x,\kappa)$ (on some open set) deserves closer attention. By Berman and Plemmons \cite{berman}, Ch.~6, Thm.~4.16 such matrices are either invertible or have one dimensional kernel. Going back to the argument in the proof of Lemma \ref{lem:compartment}(a), we see that the latter can happen only if all columns of $N_{22}$ contain an entry $1$ and an entry $-1$, but this means that $\sum_{i=1}^P z_i$ is a linear first integral, hence the hypothesis of Lemma \ref{lem:compartment}(c) is satisfied.
\end{itemize}

 We now require explicitly that in the situation of Lemma \ref{lem:compartment}(c), \textit{all} linear forms $\lambda_1,\ldots,\lambda_k$ define first integrals of the system \eqref{eq:sys1} and depend on $z$ alone.  Furthermore, we require that they are {\em induced by stoichiometry}, that is, they are defined by vectors in the left kernel of $N$. This situation is quite common for chemical and biochemical reaction networks \cite{feliu:intermediates,Fel_elim,saez_reduction}. In particular $\widetilde{\Omega}_\kappa=\Omega_\kappa$ in Lemma \ref{lem:compartment}(c). 
By considering the coefficients of the linear forms, we may thus write 
\begin{equation*}
\begin{pmatrix}\lambda_1\\ \vdots\\ \lambda_k\end{pmatrix}=W\in \mathbb \R^{k\times P},
\end{equation*}
such that $W N_{21}= W N_{22}=0$.
The following was shown in \cite{Fel_elim,saez_reduction}.
\begin{lemma}\label{lem:compstoich}
One may choose the $\lambda_i$ with pairwise disjoint support and coefficients $0$ and $1$ only. Thus, up to reordering of the $z_j$ one may assume 
\[
W=\begin{pmatrix}W^\prime & I_k\end{pmatrix}  \in  \R^{k\times P},
\]
and any level set $Wz=\alpha\in \mathbb R^k_{\ge 0}$ may be rewritten in the form
\[
\begin{pmatrix}z_{p+1}\\ \vdots\\ z_P\end{pmatrix}=\begin{pmatrix}\alpha_{1}\\ \vdots\\ \alpha_k\end{pmatrix}- W^\prime\begin{pmatrix}z_{1}\\ \vdots\\ z_p\end{pmatrix}. 
\]
Moreover, for any $\alpha\in\mathbb R_{\geq 0}^k$ and {$x$ in $\Omega_\kappa$} (cf.\ Lemma~\ref{lem:compartment}(c)) the 
linear system in $z$
\[
\alpha = Wz, \qquad 0 = N_{21} v_1(x,\kappa) + K_{2}(x,\kappa)z
\]
has a unique solution,  which is non-negative. 
\end{lemma}

\begin{proposition}\label{prop:crsredprop}
Consider the situation of Lemma \ref{lem:compartment}(c)  with all linear forms induced by stoichiometry and $W,W'$ as in  Lemma \ref{lem:compstoich}.
Denote by $\widetilde K_2$ the $p\times P$-matrix containing the first $p$ rows of $K_2$, by $\widetilde N_{21}$ the matrix containing the first $p$ rows of $N_{21}$, and partition 
\[
K_1=\begin{pmatrix}K_{11}&K_{12}\end{pmatrix},\quad \widetilde K_2=\begin{pmatrix}\widetilde K_{21}&\widetilde K_{22}\end{pmatrix}
\]
into matrices with $p$ resp. $k=P-p$ columns. 
 Then for any $\alpha\in\mathbb R_{\geq 0}^k$ the restriction of \eqref{eq:sys1} to the level set $Wz=\alpha$ induces the following system in $\mathbb R^{n+p}_{\geq 0}$:
\begin{align}\label{eq:reformulate}
\dot x&=N_{11}v_1(x,\kappa)+K_{12}(x,\kappa)\alpha +\left(K_{11}(x,\kappa)-K_{12}(x,\kappa)W^\prime\right)z_{1:p}\\
\dot{ z}_{1:p} &=\widetilde N_{21}v_1(x,\kappa)+\widetilde K_{22}(x,\kappa)\alpha+\left(\widetilde K_{21}(x,\kappa)-\widetilde K_{22}(x,\kappa)W^\prime\right) z_{1:p}, \nonumber
\end{align}
where $z_{1:p}=(z_1,\ldots,z_p)$.
\end{proposition}

\begin{proof}
This follows from replacing $(z_{p+1},\ldots, z_P)$ in \eqref{eq:sys1} by way of Lemma \ref{lem:compstoich}.
\end{proof}

In the following ve take $V_1=\R^n_{>0}$ and $V_2=\R^p_{>0}$.
Consider  a curve in the  joint parameter space of $\kappa$ and $\alpha$, $(\kappa,\alpha)=(\widehat\kappa,\widehat\alpha)+\epsilon(\kappa^*,\alpha^*)+\ldots\in\Pi\times R^{k}_{\ge0}.$ 
Then system \eqref{eq:reformulate} can be written in the form of  \eqref{epsilonsystem} with
\begin{equation}\label{eq:ab1}
a_0(x)=N_{11}v_1(x,\widehat\kappa)+K_{12}(x,\widehat\kappa)\widehat\alpha,\quad A_0(x)=K_{11}(x,\widehat\kappa)-K_{12}(x,\widehat\kappa)W^\prime
\end{equation}
\begin{equation}\label{eq:ab2}
b_0(x)=\widetilde N_{21}v_1(x,\widehat\kappa)+\widetilde K_{22}(x,\widehat\kappa)\widehat\alpha,\quad B_0(x)=\widetilde K_{21}(x,\widehat\kappa)-\widetilde K_{22}(x,\widehat\kappa)W^\prime.
\end{equation}

\begin{proposition}\label{lem:consequences}
Assume that notation and hypotheses are as in Proposition \ref{prop:crsredprop}. 
\begin{enumerate}[(a)]
\item If $\widehat\kappa$ is such that $K_2(x,\widehat\kappa)$ has rank $p$ in $\Omega_{\widehat\kappa}$, then $B_0(x)=\widetilde K_{21}(x,\widehat\kappa)-\widetilde K_{22}(x,\widehat\kappa)W^\prime$ is invertible and all its eigenvalues have negative real part. In particular blanket condition (i) is satisfied. 
\item If furthermore $v_1(x,\widehat\kappa)=0$ and $\widehat\alpha=0$, then $a_0(x)=0$ and $b_0(x)=0$, and blanket condition (ii) is satisfied. By (a) and Lemma \ref{lem:lessuglylemma}(b), $M=I_p$ and a Tikhonov-Fenichel reduction with a linearly attractive slow manifold exists. Furthermore $w=0$, hence Propositions \ref{zeezeroprop} and \ref{qssprop}  apply and  the singular perturbation reduction agrees with the classical QSS reduction. 
\item If blanket condition (ii) is satisfied and $w$ is constant, then $\Delta=B_0$ and a Tikhonov-Fenichel reduction with a linearly attractive slow manifold exists. By Corollary~\ref{qsscor}, the singular perturbation reduction agrees with the classical QSS reduction. 
\end{enumerate}
\end{proposition}
\begin{proof}
The first statement of part (a) is just a reformulation of Lemma \ref{lem:compartment}(c). The remaining assertions are clear.
\end{proof}

It was shown in \cite{saez_reduction} that the classical QSS reduction system can be interpreted as the ODE system associated with a reduced network in the species in $\mX$ with an appropriate choice of kinetics.   We will not go further into this here.

When the rate functions are multiple of some entry of $\kappa$, it follows from system \eqref{eq:reformulate} and the assumptions and notation of  system \eqref{epsilonsystem}  that
\begin{align*}
 a_1(x) & =N_{11}v_1(x,\kappa^*)+K_{12}(x,\kappa^*)\widehat\alpha+K_{12}(x,\widehat\kappa)\alpha^* & A_1(x) & =K_{11}(x,\kappa^*)-K_{12}(x,\kappa^*)W^\prime,  \\
 b_1(x) & =\widetilde N_{21}v_1(x,\kappa^*)+\widetilde K_{22}(x,\kappa^*)\widehat\alpha+\widetilde K_{22}(x,\widehat\kappa)\alpha^*, &  B_1(x)& =\widetilde K_{21}(x,\kappa^*)-\widetilde K_{22}(x,\kappa^*)W^\prime,
\end{align*}
hence $A_0,A_1$, resp.~$B_0,B_1$, are the same functions evaluated in different parameter points (c.f. \eqref{eq:ab1}, \eqref{eq:ab2}).

\medskip
To conclude this section, we illustrate that some of the first linear integrals from Lemma~\ref{lem:compartment} may depend on $x$ and $\k$ in some situations, and therefore may not all be induced by stoichiometry. 

\begin{example}\label{ex:1}
 Consider the network $Z_2\ce{<-[\kappa_1]} Z_1\ce{->[\kappa_2]} Z_3$ with only non-interacting species and mass-action kinetics. The matrix $K_2(x,\kappa)$ is found from
$$N v_2(x,z,\kappa)=\begin{pmatrix} -1 & -1 \\ 1& 0\\ 0&1\end{pmatrix} \begin{pmatrix} \kappa_1z_1\\ \kappa_2z_1\end{pmatrix}=\begin{pmatrix} -(\kappa_1+\kappa_2) & 0 & 0 \\ \kappa_1 & 0 & 0\\ \kappa_2 &0&0\end{pmatrix}z.$$
This matrix vanishes when evaluated at the two  linear  forms $z_1+z_2+z_3$ and  $\kappa_3z_1-\kappa_2z_2$. Both of these forms are independent of $x$, but only the first is independent of $\kappa$.
Similarly, consider the network $Z_1\ce{<-[\kappa_1]} X_1\ce{->[\kappa_2]} Z_2$ with two non-interacting species $Z_1,Z_2$ and mass-action kinetics.  Now $m_2=0$ and the matrix $K_2(x,\kappa)z$ is  obtained from
$$N_{22} v_2(x,z,\kappa)=\begin{pmatrix}  0 & 0\\0&0\end{pmatrix}z.$$
(Note that $N_{22}$ is a $2\times 0$ matrix and $v_2(x,z,\kappa)$ is a $0\times 1$ matrix.)
Hence, this matrix vanishes when evaluated at any linear form. 
However, the ODE system admits only one independent linear first integral in $z_1,z_2$, namely $\kappa_3z_1-\kappa_2z_2$, which depends on the choice of reaction rate constants.
\end{example}


\section{The non-interacting graph}

In this section we relate the results of the previous section to a particular labelled multi-digraph built from the reaction network and the set of non-interacting species, thus extending the formalism introduced in  \cite{saez_reduction} from steady state to quasi-steady state. The two blanket conditions may be interpreted in terms of conditions on this graph, which (at least for relatively small networks) allows for easy identification of  TFPVs.

We keep the special designations for $\Pi$, $\Omega_\kappa$ etc.\ from Section \ref{crnsec}. Recall that the vector of   rate functions takes the form $v(x,z,\kappa)=(v_1(x,\kappa),v_2(x,z,\kappa))$ and that $v_2(x,z,\kappa)$ is linear in $z$, with each component depending on one $z_j$. 
Therefore, we write 
\[ v_2(x,z,\k)_i = \nu_2(x,\k)_i z_j,\quad j=j(i),\]
if $Z_j$ is the non-interacting species in the reactant of the considered reaction $r_{m_1+i}$.
Recall that the rate functions are evaluated only in the non-negative orthant. 
We now decide for what values of the parameters $\k$  the blanket conditions (i) and (ii) are satisfied.

Given $\k$, we follow \cite{saez_reduction} and introduce a labelled multi-digraph $\mG_{\k}=(\mN,\mE_{\k})$ describing the part of the network relating to the non-interacting species only. The node set is
$$\mN=\{Z_1,\ldots,Z_P,*\}$$
and the edge set $\mE_{\k}$ is defined by the following edges and labels, for each reaction $r_i$, $i=1,\ldots,m$,
\begin{align*}
Z_j & \ce{->[\nu_2(\cdot ,\k)_i]} Z_k && \text{if $r_{m_1+i}$ involves $Z_j$ in the reactant and $Z_k$ in the product,} \\ 
Z_j& \ce{->[\nu_2(\cdot,\k)_i]} * &&\text{if $r_{m_1+i}$ involves $Z_j$ in the reactant and no non-interacting  species in  the product,} \\
*& \ce{->[v_1(\cdot,\k)_i]} Z_k &&\text{if $r_i$ involves no non-interacting species in the reactant and $Z_k$ in  the product.}
\end{align*} 
The labels are functions of $x$. We let $\ell_{\kappa}(e)$ denote the label of a given edge $e$, which in turn is a function of $x$. 
For a submulti-digraph $\mG'=(\mN',\mE')$ of $ \mG_{\k}$,  we define the label of $\mG'$ by 
$$\ell_{\kappa}(\mG')=\prod_{e\in\mE'} \ell_{\kappa}(e).$$

Let $\mG_{\kappa}=\mG_{\kappa}^0\cup \mG_{\kappa}^1\cup \ldots\cup \mG_{\kappa}^d$ be the partitioning of $\mG$ into its connected components $\mG^i_{\kappa}=(\mN^i,\mE_{\kappa}^i)$, such that $\mG_{\kappa}^0$ is the component containing the node $*$. 
The component $\mG_{\kappa}^0$ consists of only the node $*$ if all edges of $\mG_{\kappa}$ are between two non-interacting species. Since all species of the network are in at least one reaction (by assumption), all non-interacting species nodes will be connected to at least one other node, potentially $*$. Therefore a connected component cannot consist of only one non-interacting species.  For each connected component $\mG_{\kappa}^i$, $i=1,\ldots,d$, there is a corresponding first linear integral $\lambda_i$ (as in Lemma \ref{lem:compstoich}) with coefficient one for the entries corresponding to the nodes $V\in\mG^i_{\kappa}$ and zero otherwise \cite{Fel_elim}. 
Hence $d\leq k$ in Lemma~\ref{lem:compartment}(c) and our assumption that all linear forms are induced by stoichiometry imposes $d=k$. 
Let $\alpha_i=\sum_{j=1}^P \lambda_{ij} z_j$ be the conserved amount.

Furthermore, let $\Theta_{\kappa,i}(V)$, $i=0,1,\ldots,d$, be the set of spanning trees rooted at the node $V\in\mN^i$. 
To be precise,  the edges of a spanning tree are directed towards $V$, and there is precisely one outgoing edge for each $V'\in\mN^i$, except for the root $V$. Furthermore, the set of spanning trees which have positive labels when evaluated for $x\in \Omega_\kappa\cap\mathbb R^n_{\geq 0}$  is denoted as
$$\Theta_{\kappa,i}^+(V)=\{\tau\,|\,\tau \in\Theta_{\kappa,i}(N),\, \ell_\kappa(\tau)>0 \textrm{ in }\Omega_\kappa\cap\mathbb R^n_{\geq 0}\}.$$

Next we relate the blanket conditions to conditions on the graph. For convenience we consider the joint parameter space  of $\kappa$ and $\alpha$, and let $(\kappa,\alpha)=(\widehat\kappa,\widehat\alpha)+\epsilon(\kappa^*,\alpha^*)+\ldots$ with  $\widehat\kappa\in \Pi$, $\widehat\alpha\in \R^d_{\ge0}$ be a curve in the joint parameter space for $\epsilon\ge 0$. Consider the ODE system 
\begin{align*}
\dot x&=N_{11}v_1(x,\widehat\kappa)+K_{12}(x,\widehat\kappa)\widehat\alpha +\big(K_{11}(x,\widehat\kappa)-K_{12}(x,\widehat\kappa)W^\prime\big)z_{1:p},\\
\dot{z}_{1:p} &=\widetilde N_{21}v_1(x,\widehat\kappa)+\widetilde K_{22}(x,\widehat\kappa)\widehat\alpha+\big(\widetilde K_{21}(x,\widehat\kappa)-\widetilde K_{22}(x,\widehat\kappa)W^\prime\big) z_{1:p}, \nonumber
\end{align*}
or in the notation of \eqref{eq:ab1} and \eqref{eq:ab2},
$$\dot x= a_0(x)+A_0(x)z_{1:p},\qquad \dot{z}_{1:p}=b_0(x)+B_0(x)z_{1:p}.$$

The next lemma tells us that blanket condition (i) corresponds to the existence of at least one rooted spanning tree with positive label in each connected component, and the root must be $*$ for the component $\mathcal{G}_{\k}^0$.

\begin{lemma}\label{lem:spanning}
 $B_0(x)$ is invertible for a fixed $\widehat\kappa\in \Pi$ and $x\in \Omega_{\widehat\kappa}\cap\mathbb R^n_{\geq 0}$ if and only if $\Theta_{\widehat\kappa,0}^+(*)\not=\emptyset$ and  $\cup_{V\in\mN^i}\Theta_{\widehat\kappa,i}^+(V)\not=\emptyset$ for all $i=1,\ldots,d$. 
\end{lemma}
\begin{proof}
We have $B_0(x)=\widetilde K_{21}(x,\widehat\kappa)-\widetilde K_{22}(x,\widehat\kappa)W^\prime$. It is shown in \cite{Fel_elim,saez_reduction} (with a proof based on the Matrix-Tree theorem) that invertibility is equivalent to the condition of the lemma. 
\end{proof}

If at least one of the sets of spanning trees described in the lemma is empty, then there are additional conservation relations among non-interacting species, as in Example \ref{ex:1}, or $\widehat{\k}$ is such that some reactions have vanishing rate, and hence they are not present in practice. 

\medskip
We now proceed to address blanket condition (ii). To this end, we need to introduce some extra objects. 
Let $\sigma$ be a cycle of the graph $\mG$, say in the connected component $\mG_i$, $e$ an edge of $\sigma$, and define
$$\Gamma(\sigma)=\{\tau\cup \sigma\,|\,   \tau\in \Theta_{\widehat\kappa,i}(\textrm{source of }e) \textrm{ and }\sigma\setminus e\,\,\text{is a subgraph of}\,\, \tau\},$$ 
where $\cup$ and $\setminus$ are applied to both node set and edge set.
That is, $\Gamma(\sigma)$ consists of spanning trees that, after the addition of an edge, contain the cycle $\sigma$. 
It is shown in \cite{saez_reduction} that $\Gamma(\sigma)$ does not depend on the choice of $e$. We consider now the set  $\Sigma$ of the cycles $\sigma$ of $\mG_\k$ such that $\Gamma(\sigma)\not=\emptyset$ and further the sum of the columns of the stoichiometric matrix $N$ corresponding to the reactions in the cycle does not vanish on the $x$-coordinates. That is, if $\zeta_\sigma\in \R^n$ denotes the projection onto $\R^n$ of the  sum of the reaction vectors of the reaction in $\sigma$, the cycle $\sigma$ belongs to $\Sigma$ if and only if 
$$\zeta_\sigma\not=0\quad\text{and}\quad \Gamma(\sigma)\not=\emptyset.$$
 The first condition means that the net production of the non-interacting species is non-zero in the reaction path composed of the reactions in  the cycle.
A cycle consisting of two reactions forming one reversible reaction never satisfies this condition as the sum would be zero. 
We let $\Sigma_0,\Sigma_1,\dots,\Sigma_d$ denote the respective subsets in each connected component of $\mathcal{G}_\k$.

Let $I\subseteq\{1,\ldots,m_1\}$ be the set of indices of the reactions that do not involve any non-interacting species, that is, $\sum_{j=1}^P \delta_{ij}=\sum_{j=1}^P \delta'_{ij}=0$ for $i\in I$.
For fixed $\k,\alpha$, and under blanket condition (i), it  is shown  in \cite{saez_reduction} that  
in $\Omega$ the following equality holds
\begin{equation}\label{eq:h}
a_0(x)-A_0(x)B_0(x)^{-1}b_0(x)=\sum_{i\in I} v_{1}(x,\k)_i \xi_i + \sum_{i=0}^{d}  \frac{\alpha_{i}}{q_{i}(x,\kappa)} \sum_{\sigma\in\Sigma_i} \left( \sum_{\gamma\in \Gamma(\sigma)}  \ell_\k (\gamma) \right) \zeta_\sigma,
\end{equation}
where 
\begin{itemize}
\item $\alpha_0=1$ for convenience,
\item $\xi_i\in \R^n$ is the vector with entries $\xi_{ij}=\beta'_{ij}-\beta_{ij}$ (the net production of the species  in $\mX$ in reaction $r_i$),  
\item $\ell_\k(\gamma)$ is the label of the subgraph $\gamma$ of $\Gamma(\sigma)$, and has $\ell_\k(\sigma)$ as a factor,
\item 
the function $q_{i}(x,\k)$ is positive if blanket condition (i) is satisfied. In particular, it is the sum of the labels of the trees in $\Theta_{\widehat\kappa,0}^+(*)$ for $i=0$ and of the labels of the trees in  $\cup_{V\in\mN^i}\Theta_{\widehat\kappa,i}^+(V)$ for all $i=1,\ldots,d$ (c.f. Lemma~\ref{lem:spanning}). 
\end{itemize}
We remark that in \cite{saez_reduction}, it is assumed the parameter $\widehat\kappa$ is positive, but this is not necessary as long as  blanket condition (i) holds.

Using equality \eqref{eq:h}, we see that blanket condition (ii) holds if and only if the right hand side of \eqref{eq:h} vanishes. 
In the next lemma we obtain a sufficient condition for this to occur. 

\begin{lemma}\label{lem:cycle}
Assume blanket condition (i) is satisfied for a fixed  $\widehat\kappa\in\Pi$, and let $\widehat\alpha\in\R^d_{\ge 0}$.
A sufficient condition for  blanket condition (ii) to be satisfied is that:
\begin{itemize}
\item[(a)]  $v_{1}(x,\widehat\k)_i=0$ for all $i\in I$, 
\item[(b)] $\sum_{\gamma\in \Gamma(\sigma)}  \ell_\k (\gamma) =0$ for all cycles $\sigma\in\Sigma_0$  of $\mG^0_{\k}$,
\item[(c)] $\widehat\alpha_i \left( \sum_{\gamma\in \Gamma(\sigma)}  \ell_\k (\gamma) \right)=0$ for all cycles $\sigma\in\Sigma_i$  of $\mG^i_{\k}$, $i=1,\ldots,k$.
\end{itemize}
These conditions are necessary if the vectors $\xi_i$ for $i\in I$ and $\zeta_\sigma$ for all $\sigma\in \Sigma$ are linearly independent. 

\smallskip
Sufficient conditions for (b) and (c) to hold are 
\begin{itemize}
\item $\ell_{\widehat{\k}}(\sigma)=0$ if $\sigma\in\Sigma_0$ is a cycle of $\mG^0_{\k}$,
\item $\widehat\alpha_i \ell_{\widehat{\k}}(\sigma)=0$ if $\sigma\in\Sigma_i$ is a cycle of $\mG^i_{\k}$, $i=1,\ldots,k$.
\end{itemize}
\end{lemma}

\begin{proof}
The statement is a consequence of the form of the right hand side of \eqref{eq:h} 
as all terms vanish under the conditions of the lemma. For the second part, we note that  $\ell_{\widehat\k}(\gamma)$ is a multiple of $\ell_{\widehat\k}(\sigma)$. 
 \end{proof}

In the special case of mass-action kinetics, or kinetics for which each of the rate functions is multiplied by one of the parameters,  the first two conditions in Lemma~\ref{lem:cycle} hold if the corresponding parameters  are set to zero. Specifically, we obtain the following corollary, which is a consequence of Proposition~\ref{lem:consequences} and Lemmas~\ref{lem:spanning} and \ref{lem:cycle}.

\begin{corollary}\label{cor:QSS}
Assume $\Pi=\Pi_1\times \Pi_2\subseteq \R^{m}_{\geq 0} \times \R^{q-m}$ such that 
\begin{align*}
v_1(x,\k)_i &= \k_i u_1(x,\k')_i &&\textrm{for all }i=1,\dots,m_1\\ 
\nu_2(x,\k)_i & =\k_i u_2(x,\k')_i && \textrm{for all }i=m_1+1,\dots,m,
\end{align*}
with $u_1,u_2$ functions of $x$ and $\k'\in \Pi_2$ taking only positive values. Assume further  that   $\Theta_{\kappa,0}(*)\not=\emptyset$ and  $\cup_{V\in\mN^i}\Theta_{\kappa,i}(V)\not=\emptyset$ for all $i=1,\ldots,d$ (whether these hold does not depend on $\k$).
\begin{itemize}
\item
Let $\widehat\k\in \Pi_1\times \Pi_2$ such that $\widehat\kappa_{1:m_1}=0$, $\widehat\k_{i}>0$ for $i=m_1+1,\dots,m$, and let $\widehat\alpha=0$. Then blanket conditions (i) and (ii) are satisfied and furthermore, $w(x)=0$ and all eigenvalues of $\Delta(x)=B_0(x)$ have negative real part for $x\in \Omega_\kappa\cap\mathbb R^n_{\geq 0}$. Consequently, a Tikhonov-Fenichel reduction exists and agrees with the QSS reduction.
\item 
In particular,  there is a choice of parameters for which the QSS reduction can be seen as a Tikhonov-Fenichel reduction of the original system. 
\end{itemize}
\end{corollary}
\begin{proof}
Any tree in  $\Theta_{\widehat\kappa,i}(V)$, for any $i,V$ with $V=*$ if $i=0$, has only edges with source an element in $Z$.
These edges have label of the form $\widehat\k_j u_2(x,\widehat\k')_j $ with $\widehat\k' \in  \Pi_2$, which by assumption is strictly positive. 
Hence, any spanning tree in the relevant sets $\Theta_{\widehat\kappa,i}(V)$ has positive label. By Lemma~\ref{lem:spanning}, blanket condition (i) holds.

Conditions (a) and (c) from Lemma~\ref{lem:cycle} hold trivially. Consider now a cycle $\sigma\in \Sigma_0$. If the cycle contains the node $*$,  then the label of the edge with source $*$ is zero, and hence $\ell_{\widehat\k}(\sigma)=0$. If $*$ is not in the cycle, then consider any subgraph $\gamma \in \Gamma(\sigma)\neq \emptyset$. This subgraph contains a spanning tree with root a node of the cycle. Hence, it must contain an edge with source $*$, which has zero label, implying that $\ell_{\widehat{\k}}(\gamma)=0$ for any $\gamma \in \Gamma(\sigma)$. It follows that condition (b) of Lemma~\ref{lem:cycle} is satisfied as well, and hence blanket condition (ii) holds.
\end{proof}

A nice consequence of the Corollary is that if all reactions involve some species in $Z$ in the reactant, then $\widehat{\alpha}=0$ defines a TFPV, regardless of the (positive) values of the reaction rate constants. 

Note that if $\mathcal{G}_\k^0$ has only the node $*$, then $b_0=0$ and hence $w=0$. 

\begin{remark}\label{rk:intermediates}
A special scenario occurs for so-called \emph{intermediate species} \cite{feliu:intermediates}: these are species that do not interact with any other species, and are the reactant and the product of at least one reaction. The set of these species is obviously a set of non-interacting species. With mass-action kinetics, $K_2(x,\k)$ has full rank and hence there are no linear first integrals in their concentrations. In particular,
\[ b_0(x)=N_{21} v_1(x,\k),\qquad B_0(x)= K_2(x,\k), \]
and $B_0(x)$ is constant in $x$. Hence, if blanket conditions (i) and (ii) are satisfied, we need only to choose $\widehat\k$ such that the rate of any reaction producing an intermediate is constant in order to obtain a valid Tikhonov-Fenichel reduction which further agrees with the QSS reduction (up to irrelevant terms of higher order in $\epsilon$). 

Let us look at this scenario in more detail. 
By the condition on the production and degradation of all intermediates, the graph $\mathcal{G}_\k$ has one connected component, namely that of $*$, which necessarily has a spanning tree rooted at $*$. 
The label of the spanning tree may be zero depending on $\widehat\kappa$. Hence blanket condition (i) is satisfied if and only if there is a  directed path  from any intermediate species to $*$ with positive label.

The cycles of $\mG_\k$ are of two kinds.  A cycle is not in $\Sigma$ if it  does not go through $*$,  because the reactions corresponding to the cycle only involve non-interacting species. If a cycle goes through $*$, then it contains an edge of the form
$*\ce{->} Z.$
By setting the reaction rate constant of all reactions of this form to zero, we are guaranteed that \eqref{eq:h} is zero, that is, blanket condition (ii) holds. 
This straightforwardly implies that $b_0(x)=0$, hence also $w=0$. Hence by Corollary \ref{cor:QSS} there exists a Tikhonov-Fenichel reduction and it agrees with the QSS reduction. By Proposition \ref{zeezeroprop} the reduced system is
$$\frac{dx}{d\tau}=a_1(x)-A_0B^{-1}_0b_1(x).$$
By the nature of the reactions, the matrices $A_0$  and $B_0$ are constant in the concentrations $x$.
\end{remark}

\medskip

Before moving to the discussion of realistic examples in the next section, we provide an illustrative example to show that the conditions in Lemma \ref{lem:cycle} are sufficient but not necessary. 

\begin{example}
Consider the (artificial) reaction network
\begin{equation*}
X_1+Z_1\ce{->[\kappa_1]}  2X_1, \quad X_1\ce{->[\kappa_2]}  2X_1+Z_1, \quad X_1+Z_1\ce{->[\kappa_3]} 0, \quad X_1\ce{->[\kappa_4]} Z_1
\end{equation*}
with $\mZ=\{Z_1\}$, $\mX=\{X_1\}$, and assuming mass-action kinetics. The graph $\mathcal{G}_\k$ for $\widehat\kappa$ is
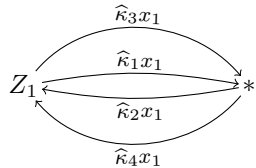
\begin{center}
\begin{tikzpicture}[inner sep=1.2pt]
\node (Z) at (0,0) {$Z_1$};
\node (*) at (3,0) {$*$};
\draw[->] (Z) to[out=10,in=170] node[above,sloped]{\footnotesize $\widehat\kappa_1x_1$} (*);
\draw[->] (Z) to[out=50,in=130] node[above,sloped]{\footnotesize $\widehat\kappa_3x_1$} (*);
\draw[->] (*) to[out=-170,in=-10] node[below,sloped]{\footnotesize $\widehat\kappa_2x_1$} (Z);
\draw[->] (*) to[out=-130,in=-50] node[below,sloped]{\footnotesize $\widehat\kappa_4x_1$} (Z);
\end{tikzpicture}
\end{center}
It has exactly two spanning trees rooted at $*$, namely $Z_1\ce{->[\widehat\kappa_1x_1]}*$ and $Z_1\ce{->[\widehat\kappa_3x_1]}*$, so either of these two coefficients must be non-zero for blanket condition (i) to be fulfilled, see Lemma \ref{lem:spanning}. 
There are four possible cycles, but only two are in $\Sigma$, namely 
$$\sigma_1\colon\quad Z_1\ce{->[\widehat\kappa_1x_1]}*\ce{->[\widehat\kappa_{2}x_1]}Z_1,\qquad \sigma_2\colon  \quad Z_1\ce{->[\widehat\kappa_3x_1]}*\ce{->[\widehat\kappa_{4}x_1]}Z_1.$$
For each cycle $\sigma_i$, $\Gamma(\sigma_i)$ contains only $\sigma_i$. Furthermore, $\zeta_{\sigma_1}= 2$ and $\zeta_{\sigma_2}=-2$. 
The other two cycles have net production of $X_1$ equal to zero. We have $I=\emptyset$ and the function on the right side of \eqref{eq:h} is
\[ \frac{1}{q(x_1,\widehat\kappa)}\big(2 \widehat\kappa_1\widehat\kappa_2 x_1^2-2\widehat\kappa_3\widehat\kappa_4 x_1^2 \big) = \frac{2x_1^2}{q(x_1,\widehat\kappa)}\big( \widehat\kappa_1\widehat\kappa_2 -\widehat\kappa_3\widehat\kappa_4 \big) ,
\]
where $q(x_1,\widehat\kappa)$ is positive if blanket condition (i) holds, that is, if at least one of $\widehat\k_1,\widehat\k_3$ are positive. 
 
Blanket condition (ii) is fulfilled by choosing for example $\widehat\kappa_{2}=\widehat\kappa_{4}=0$ according to Lemma \ref{lem:cycle}. However, it is clear that the function also vanishes if $\widehat\kappa_1\widehat\kappa_2-\widehat\kappa_3\widehat\kappa_4=0$. This implies the conditions in Lemma \ref{lem:cycle} are only sufficient and not necessary.  

To complete the example, we note that
\begin{align*}
a_0(x_1) &= (\widehat\kappa_2 - \widehat\kappa_4) x_1,  &  A_0(x_1)= \phantom{-}(\widehat\kappa_1 - \widehat\kappa_3) x_1,\\
b_0(x_1) & = ( \widehat\kappa_2+ \widehat\kappa_4) x_1 & B_0(x_1)=  -(\widehat\kappa_1 + \widehat\kappa_3) x_1,
\end{align*} 
hence 
$$w(x_1)=B_0(x_1)^{-1}b_0(x_1)=-\frac{\widehat\kappa_2+\widehat\kappa_4}{\widehat\kappa_1+\widehat\kappa_3},\qquad \Delta(x_1)=B_0(x_1)=-(\widehat\kappa_1 + \widehat\kappa_3) x_1,$$
and a Tikhonov-Fenichel reduction exists according to Lemma \ref{lem:lessuglylemma}(b). Finally, according to Corollary \ref{qsscor}, the Tikhonov-Fenichel reduction and the QSS reduction agree since $w$ is constant.
\end{example}

\section{Examples and applications}

\subsection{The Michaelis-Menten mechanism}
\label{sec:MM}

For the purpose of illustration, we will discuss the standard enzyme-substrate mechanism for some choices of non-interacting sets with mass-action kinetics. (Note that all possible QSS and Tikhonov-Fenichel reductions of this system are discussed in \cite{gwz,gwz3}.) The mechanism is
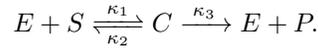
\begin{equation*}
E+S\ce{<=>[\kappa_1][\kappa_{2}]}  C\ce{->[\kappa_3]} E+P.
\end{equation*}
 There are two linear first integrals, which are given by the stoichiometry, namely $x_E+x_C$ and $x_C+x_S+x_P$.  The associated ODE system is
\begin{align*}
\dot{x}_E &= -\k_1 x_E x_S + (\k_2+\k_3) x_C \\
\dot{x}_C &= \k_1 x_E x_S - (\k_2+\k_3) x_C \\
\dot{x}_S &= -\k_1 x_E x_S + \k_2 x_C \\
\dot{x}_P &= \k_3 x_C.
\end{align*}
The domain $\Omega_{\widehat\kappa}$ depends on the choice of parameters, but one will always have $\mathbb R^n_{>0}\subseteq\Omega_{\widehat\kappa}$.

\paragraph{\bf Case $\mZ=\{S\}$.} The non-interacting graph for a parameter value $\widehat\kappa$ is
$$S\ce{<=>[\widehat\kappa_1x_E][\widehat\kappa_{2}x_C]}*.$$
There is one rooted spanning tree at $*$, namely $S\ce{->[\widehat\kappa_1x_E]}*$, and by Lemma \ref{lem:spanning}, blanket condition (i) holds if and only if  $\widehat\kappa_1>0$.
For   blanket condition (ii)  we consider the cycles of the graph. There is only one cycle 
$$S\ce{->[\widehat\kappa_1]}*\ce{->[\widehat\kappa_{2}]}S$$
 which is not in $\Sigma$. Here $I=\{3\}$ and the rate function of the third reaction is $\widehat\k_3 x_C$. Hence, by Lemma~\ref{lem:cycle}, blanket condition (ii) holds if and only if $\widehat\kappa_3=0$.

Now let $\widehat\k_1>0$ and $\widehat\k_3=0$. We verify the reducibility conditions from Lemma~\ref{lem:lessuglylemma}(b),
and we find 
$$B_0(x)=-\widehat\kappa_1 x_E,\quad w(x)=B_0(x)^{-1}b_0(x)=-\frac{\widehat\kappa_2 x_C}{\widehat\kappa_1 x_E}.$$
If $\widehat\kappa_2=0$, then  $\Delta(x)=B_0$, hence a Tikhonov-Fenichel reduction exists and agrees with QSS, according to Lemma \ref{lem:lessuglylemma}(b) and Corollary \ref{qsscor}. If $\widehat\kappa_2>0$, then
$$\Delta(x)=-\left(\widehat\kappa_1x_E+\frac{\widehat\kappa_2x_C}{x_E}+\widehat\kappa_2\right),$$
 and a Tikohnov-Fenichel reduction exists according to Lemma \ref{lem:lessuglylemma}(b); note that here $M\not=\Delta$. However, the reduction agrees with the QSS reduction only in a degenerate setting: By Proposition  \ref{zeezeroprop}, we find that for the two reductions to agree, identity \eqref{agreecond} must be satisfied. With
 $$A_0B_0^{-1}(B_1-b_1)=\begin{pmatrix} -1 \\1 \\0\end{pmatrix} \left(\frac{\widehat\kappa_2}{\widehat\kappa_1}\kappa_1^*-\kappa_2^*\right)x_C,\quad A_1w-a_1=\begin{pmatrix} -1 \\1 \\0\end{pmatrix}\left(\frac{\widehat\kappa_2}{\widehat\kappa_1}\kappa_1^*-\kappa_2^*\right)x_C +\begin{pmatrix} 1 \\-1 \\-1\end{pmatrix}\kappa_3^*x_C,$$
this holds if and only if $\kappa_3^*=0$.  However, this implies $\kappa_3= \widehat\kappa_3+\varepsilon^2\cdots$ for the curve in parameter space, hence the reduced system is trivial, providing no information. (Moreover, if one makes the obvious choice $\kappa_3=\widehat\kappa_3+\varepsilon\kappa_3^*$, then  last reaction of the mechanism does not occur at all.)

\paragraph{\bf Case $\mZ=\{E,C\}$.} The non-interacting graph for a given $\widehat\kappa$ is 
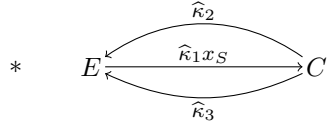
\begin{center}
\begin{tikzpicture}[inner sep=1.2pt]
\node (*) at (-1,0) {$*$};
\node (E) at (0,0) {$E$};
\node (C) at (3,0) {$C$};
\draw[->] (E) to node[above]{\footnotesize $\widehat\kappa_1x_S$} (C);
\draw[->] (C) to[out=145,in=35] node[above,sloped]{\footnotesize $\widehat\kappa_2$} (E);
\draw[->] (C) to[out=-155,in=-25] node[below,sloped]{\footnotesize $\widehat\kappa_3$} (E);
\end{tikzpicture}
\end{center}
with two connected components.
There are  three rooted spanning trees to consider, namely 
$$E\ce{->[\widehat\kappa_1x_S]} C,\quad C\ce{->[\widehat\kappa_{2}]} E,\quad C\ce{->[\widehat\kappa_3]} E,$$
so at least one of the parameters needs to be different from zero  in order to satisfy blanket condition (i) according to Lemma \ref{lem:spanning}. Furthermore, $I=\emptyset$ and there is only one cycle in $\Sigma$, viz.
$$\sigma\colon \quad E\ce{->[\widehat\kappa_1x_S]}C\ce{->[\widehat\kappa_3]} E.$$
The cycle belongs to the connected component with the linear first integral $x_E+x_C=\alpha_1$. 
The set $\Gamma(\sigma)$ contains only $\sigma$. By Lemma~\ref{lem:cycle}, 
in conjunction with blanket condition (i), blanket condition (ii) holds if and only if $\widehat\alpha_1 \widehat\kappa_1\widehat\kappa_3 x_S=0$. The case that $x_S$ vanishes identically may be dismissed, since it would amount to $\mZ=\{E,C,S\}$. We are left with the following scenarios:
\begin{itemize}
\item $\widehat\kappa_1>0$, and $\widehat\kappa_3=0$ or $\widehat\alpha_1=0$
\item $\widehat\kappa_2>0$, and $\widehat\kappa_1=0$ or $\widehat\kappa_3=0$ or $\widehat\alpha_1=0$
\item $\widehat\kappa_3>0$, and $\widehat\kappa_1=0$ or $\widehat\alpha_1=0$.
\end{itemize}
By substituting $x_C = \alpha_1 - x_E$, we are in the setting of \eqref{epsilonsystem} and \eqref{eq:ab2}, with
\[
x=\begin{pmatrix}x_S\\ x_P\end{pmatrix} \text{  and  } z=x_E,
\]
and
\begin{align*}
b_0 & =  (\widehat \k_2 + \widehat \k_3)\widehat\alpha_1, & B_0 &= - \widehat\k_1 x_S - (\widehat \k_2 + \widehat \k_3),
&a_0&=\widehat\alpha_1\begin{pmatrix} \widehat\k_2 \\  \widehat\k_3 \end{pmatrix},
 &A_0 &= \begin{pmatrix} -\widehat \k_1 x_S - \widehat\k_2 \\  -\widehat\k_3 \end{pmatrix}.
\end{align*}
Moreover 
\[
w(x)= \frac{- (\widehat \k_2 + \widehat \k_3)\widehat\alpha_1}{ \widehat\k_1 x_S + (\widehat \k_2 + \widehat \k_3)}.
\]
In the cases where $\widehat\alpha_1=0$ or $\widehat\k_2=\widehat\k_3=0$, we have  $b_0(x)=0$, so $\Delta=B_0$, and all eigenvalues lie on the left half plane. Consequently, a Tikohnov-Fenichel reduction exists and agrees with the QSS reduction.
When $\widehat\alpha_1\neq 0$ but $\widehat \k_1=0$, then $w(x)$ is constant and by Proposition~\ref{lem:consequences},  a Tikohnov-Fenichel reduction exists and agrees with the QSS reduction. 

The only scenario left to analyze is when $\widehat\k_3=0$ and the remaining parameters are positive, thus we have a curve $\varepsilon \kappa_3^*$ in parameter space. One easily checks that $\Delta(x)<0$ for all $x$. But  identity \eqref{agreecond} would imply (after some computation) that
\[
\k_3^*\left(\frac{\widehat \k_2 }{ \widehat\k_1 x_S + \widehat \k_2}-1\right)=0.
\]
This yields a contradiction unless $\k_3^*=0$, but the latter characterizes the degenerate case that the last reaction does not occur at all.

\paragraph{\bf The case $\mZ=\{P\}$.} This is a case where no Tikhonov-Fenichel reduction (and no QSS reduction) exists: The non-interacting graph for $\widehat\kappa$  is
$$*\ce{->[\widehat\kappa_3x_C]}P.$$
There is no spanning tree rooted at $*$, so blanket condition (i) cannot be satisfied.

\subsection{A predator-prey system}

The following three dimensional predator-prey system was introduced and discussed in \cite{rosenzweig} in the course of a first-principle derivation of the two dimensional  Rosenzweig-MacArthur system:
\begin{align*}
\dot{x}_B &= \phantom{-}\kappa_1 x_B(1-x_B) -\kappa_2 x_Bx_H,\\
\dot{x}_S &= -\kappa_3 x_S+\kappa_4 x_Bx_H, \\
\dot{x}_H &= \phantom{-}\kappa_3x_S-\kappa_4 x_Bx_H+\kappa_5 x_S-\kappa_6 x_H,
\end{align*}
with non-negative parameters $\kappa_1,\ldots,\kappa_6$.  Here $x_B$ stands for the abundance of  species $B$, the prey while $x_S$ resp.\ $x_H$ are abundances of the species $S$ and $H$ (resting and hunting predators). The three dimensional system is obtained  from an individual based stochastic model (see \cite{rosenzweig}, Section 2), upon scaling the abundance of prey. 
All Tikhonov-Fenichel parameter values for dimension two, and all reductions, were determined in \cite{rosenzweig} and its supplementary material.
We investigate here what types of reductions arise by means of sets of non-interacting sets and Lemmas~\ref{lem:spanning} and \ref{lem:cycle}.

 The ODE may be considered to arise from reaction networks  with mass-action kinetics in different ways. We make a choice different from \cite{rosenzweig} and consider, for example,
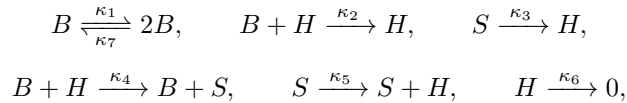
\[B\ce{<=>[\kappa_1][\kappa_7]} 2B,\qquad B+H\ce{->[\kappa_2]}H,\qquad S\ce{->[\kappa_3]} H,\]
\[ B+H\ce{->[\kappa_4]}B+S,\qquad S\ce{->[\kappa_5]} S+H,\qquad H\ce{->[\kappa_6]} 0,\]
with $\k_1=\k_7$.   

The network has only two sets of non-interacting species, namely, $\mZ=\{H\}$ and $\mZ= \{S\}$. (The union of these two sets is not a non-interacting set, and neither is $\{B\}$. According to \cite{rosenzweig} there exist QSS reductions with respect to $B$, hence our approach will not retrieve all possible QSS reductions.)

 We will provide a brief analysis of the two different sets using Lemma \ref{lem:spanning} and Lemma \ref{lem:cycle}, and compare the results to the detailed analysis carried out in \cite{rosenzweig}.   In that paper it is shown that in the case $\widehat\kappa_1=0$, $\widehat\kappa_2\widehat\kappa_3\widehat\kappa_6=0$ is a necessary condition for the existence of a Tikhonov-Fenichel reduction.  The condition is further divided into 11 cases specifying precisely the parameters that are zero and those that are not in order to obtain validity.  We will discuss these cases from the perspective of Lemma \ref{lem:spanning} and Lemma \ref{lem:cycle}, which are used without further reference.

 \paragraph{The case $\mZ=\{S\}$.} The non-interacting graph is 
 
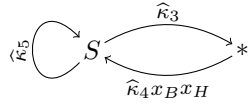
\begin{center}
\begin{tikzpicture}[inner sep=1.2pt]
\node (S) at (0,0) {$S$};
\node (*) at (2,0) {$*$};
\draw[->] (S) to[out=30,in=150]  node[above,sloped]{\footnotesize $\widehat\kappa_3$} (*);
\draw[->] (*) to[out=-150,in=-30] node[below,sloped]{\footnotesize $\widehat\kappa_4x_Bx_H$} (S);
\draw[->] (S) .. controls (-1,-1) and (-1,1).. node[above,sloped]{\footnotesize $\widehat\kappa_5$} (S);
\end{tikzpicture}
\end{center}
There is one spanning tree rooted at $*$: $S\ce{->[\widehat\kappa_3]}*$, hence only $\widehat\kappa_3>0$ is required for blanket condition (i). Additionally, there is one cycle in $\Sigma$,  $\sigma\colon S\ce{->[\widehat\kappa_5]}S$ and
$I=\{1,2,6,7\}$. Hence assuming
\begin{itemize}
\item $\widehat\kappa_3>0$,  $\widehat\kappa_1=\widehat\kappa_2=\widehat\kappa_5=\widehat\kappa_6=0$
\end{itemize}
implies that blanket conditions (i) and (ii) are satisfied. If in addition $\widehat\kappa_4=0$, then Corollary \ref{cor:QSS} applies and a Tikhonov-Fenichel reduction exists and agrees with the QSS reduction. Whenever $\widehat\kappa_4>0$, subsection 4.3 of \cite{rosenzweig} shows that a Tikhonov-Fenichel reduction exists but is not in agreement with the QSS reduction. Here we have retrieved cases 5 and 9 of  \cite[Section 3.4.2]{rosenzweig}.

If we consider \eqref{eq:h}, then $\Gamma(\sigma)$ consists of the graph 
 
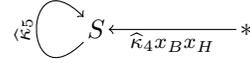
\begin{center}
\begin{tikzpicture}[inner sep=1.2pt]
\node (S) at (0,0) {$S$};
\node (*) at (2,0) {$*$};
\draw[->] (*) to node[below,sloped]{\footnotesize $\widehat\kappa_4x_Bx_H$} (S);
\draw[->] (S) .. controls (-1,-1) and (-1,1).. node[above,sloped]{\footnotesize $\widehat\kappa_5$} (S);
\end{tikzpicture}
\end{center}
which has label $\widehat\kappa_5\widehat\kappa_4x_Bx_H$. This term vanishes when $\widehat{\k}_4=0$, which gives rise to the following new set of parameters satisfying blanket conditions (i) and (ii): 
\begin{itemize}
\item $\widehat\kappa_3>0$, $\widehat\kappa_5>0$, $\widehat\kappa_1=\widehat\kappa_2=\widehat\kappa_4=\widehat\kappa_6=0$.
\end{itemize}
In this case, $w(x)=0$ and hence a Tikohnov-Fenichel reduction exists and agrees with the QSS reduction; the supplementary material to  \cite{rosenzweig} shows that the reduced system is of Volterra-Lotka type.
This is case 8 from \cite[Section 3.4.2]{rosenzweig}.

 \paragraph{The case $\mZ=\{H\}$.} The non-interacting graph is 
 
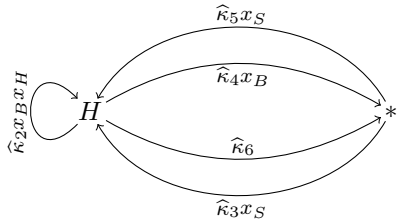
\begin{center}
\begin{tikzpicture}[inner sep=1.2pt]
\node (H) at (0,0) {$H$};
\node (*) at (4,0) {$*$};
\draw[->] (H) to[out=30,in=150]  node[below,sloped]{\footnotesize $\widehat\kappa_4x_B$} (*);
\draw[->] (H) to[out=-30,in=-150]  node[above,sloped]{\footnotesize $\widehat\kappa_6$} (*);
\draw[->] (*) to[out=-120,in=-60] node[below,sloped]{\footnotesize $\widehat\kappa_3x_S$} (H);
\draw[->] (*) to[out=120,in=60] node[above,sloped]{\footnotesize $\widehat\kappa_5x_S$} (H);
\draw[->] (H) .. controls (-1,-1) and (-1,1).. node[above,sloped]{\footnotesize $\widehat\kappa_2x_Bx_H$} (H);
\end{tikzpicture}
\end{center}
There are two spanning trees rooted at $*$: $H\ce{->[\widehat\kappa_4x_B]}*$ and $H\ce{->[\widehat\kappa_6]}*$, hence $\widehat\kappa_4>0$ or $\widehat\kappa_6>0$ are necessary and sufficient conditions for blanket condition (i).  Additionally, $I=\{1,7\}$ and there are three cycles in $\Sigma$:
$$\sigma_1\colon H\ce{->[\widehat\kappa_2x_B]}H,\qquad \sigma_2\colon H\ce{->[\widehat\kappa_6]}*\ce{->[\widehat\kappa_3x_S]} H,\qquad \sigma_3\colon  H\ce{->[\widehat\kappa_4x_B]}*\ce{->[\widehat\kappa_5x_S]} H,$$
hence, according to Lemma~\ref{lem:cycle}, the following possibilities guarantee that blanket conditions (i) and (ii) are satisfied:
\begin{itemize}
\item$\widehat\kappa_4>0$, $\widehat\kappa_1=\widehat\kappa_2=\widehat\kappa_5=\widehat\kappa_6=0$,
\item $\widehat\kappa_6>0$, $\widehat\kappa_1=\widehat\kappa_2=\widehat\kappa_3=\widehat\kappa_4=0$,
\item $\widehat\kappa_4>0$, $\widehat\kappa_6>0$, $\widehat\kappa_1=\widehat\kappa_2=\widehat\kappa_3=\widehat\kappa_5=0$.
\end{itemize}
If both $\widehat\kappa_3=\widehat\kappa_5=0$ (as in the third case), then Corollary \ref{cor:QSS} shows existence of a Tikhonov-Fenichel reduction and agreement with the QSS reduction.  If $\widehat\kappa_3>0$ in the first case or $\widehat\kappa_5>0$ in the second case, then a Tikhonov-Fenichel reduction is still obtained but it does not agree with the QSS reduction; see  \cite[Section 4.2]{rosenzweig}.  We have retrieved cases 5, 6, 7, 10 and 11 of  \cite[Section 3.4.2]{rosenzweig}.

If we consider the explicit form of \eqref{eq:h}, we find that  $\Gamma(\sigma_1)$ consists of two graphs:

\begin{minipage}[h]{0.45\textwidth}
 \begin{center}
\begin{tikzpicture}[inner sep=1.2pt]
\node (H) at (0,0) {$H$};
\node (*) at (2,0) {$*$};
\draw[->] (*) to node[above,sloped]{\footnotesize $\widehat\kappa_5x_S$} (H);
\draw[->] (H) .. controls (-1,-1) and (-1,1).. node[above,sloped]{\footnotesize $\widehat\kappa_2x_Bx_H$} (H);
\end{tikzpicture}
\end{center}
\end{minipage}
\begin{minipage}[h]{0.45\textwidth}
 \begin{center}
\begin{tikzpicture}[inner sep=1.2pt]
\node (H) at (0,0) {$H$};
\node (*) at (2,0) {$*$};
\draw[->] (*) to node[above,sloped]{\footnotesize $\widehat\kappa_3x_S$} (H);
\draw[->] (H) .. controls (-1,-1) and (-1,1).. node[above,sloped]{\footnotesize $\widehat\kappa_2x_Bx_H$} (H);
\end{tikzpicture}
\end{center}
\end{minipage}
\\
Furthermore, $\zeta_{\sigma_1}= (-1,0)^{tr}$. Hence, the term in \eqref{eq:h} corresponding to this cycle 
has numerator $-\widehat\k_2 x_B x_S (\widehat\k_3+\widehat\k_5)$ and is zero in the $x_S$ component. For the other two cycles, the sets $\Gamma(\sigma_i)$ contain only the cycle itself and when added together give the term with numerator
$(0,-\widehat\k_3 \widehat\k_6 x_S + \widehat \k_4 \widehat\k_5 x_S x_B)^{tr}$. 
The detailed analysis of the term arising from the cycle $\sigma_1$ gives the following cases for which blanket conditions (i) and (ii) hold:
\begin{itemize}
\item$\widehat\kappa_4>0$, $\widehat\kappa_1=\widehat\kappa_3=\widehat\kappa_5=\widehat\kappa_6=0$,
\item $\widehat\kappa_6>0$, $\widehat\kappa_1=\widehat\kappa_3=\widehat\kappa_5=\widehat\kappa_4=0$,
\item $\widehat\kappa_4>0$, $\widehat\kappa_6>0$, $\widehat\k_2>0$, $\widehat\kappa_1=\widehat\kappa_3=\widehat\kappa_5=0$.
\end{itemize}
This gives cases 1, 2 and 3 of \cite[Section 3.4.2]{rosenzweig}. 
In all cases $b_0=0$, hence $w=0$, and we obtain a Tikohnov-Fenichel reduction which agrees with the QSS reduction.

By our approach we could not identify case 4 of the 11 cases listed in \cite[Section 3.4.2] {rosenzweig} corresponding to setting all reaction rate constants to zero except $\widehat\kappa_2>0$. For an explanation, note that this case amounts to a QSS reduction with quasi-steady state species $B$.

\subsection{A two substrate mechanism}

We consider a mechanism that consists of  two substrates $A,B$ that are converted into two products $P,Q$ through a series of reactions catalysed by an enzyme $E$; see Cornish-Bowden  \cite[Chapter 5]{enz-kinetics}. It is an example of a \emph{bi-bi} mechanism in the notation of Cleland \cite{FH07}.
\begin{align*}
E +A&  \ce{<=>[\kappa_1][\kappa_2]} \! EA  & EA+B &\ce{<=>[\kappa_3][\kappa_4]} \! EAB  \ce{<=>[\kappa_5][\kappa_6]} \! EPQ \ce{<=>[\kappa_7][\kappa_8]} \! EQ+P & EQ &\ce{<=>[\kappa_9][\kappa_{10}]} \! E+Q. 
\end{align*}
Here, the complexes $EA, EAB, EPQ, EQ$ are seen as intermediate or transient complexes in the transformation of $A,B$ into $P,Q$.

We discuss here just one set of non-interacting species, namely $\mZ=\{Z_1,Z_2,Z_3,Z_4,Z_5\}$, where  $Z_1=E$, $Z_2=EA$, $Z_3=EAB$, $Z_4=EPQ$ and $Z_5=EQ$ are all species involving the enzyme $E$.  There is a single linear first integral relating only species in $\mZ$, which is $\lambda(z)=z_1+z_2+z_3+z_4+z_5$ ($=\alpha$). We further assume mass-action kinetics.

Let  $(\kappa,\alpha)=(\widehat\kappa,\widehat\alpha)+\epsilon(\kappa^*,\alpha^*)+\ldots$ be a curve in the joint parameter space with $\widehat\kappa\in\R_{\ge 0}^{10}$ and $\widehat\alpha\in\R_{\ge 0}$. The non-interacting graph $\mG$ for $\widehat\kappa$ is  
\begin{center}
{
\begin{tikzpicture}
\node (*) at (-2,-2) {$*$};%
\node (E) at (0,-2) {$Z_1$};
\node (EA) at (2,-2) {$Z_2$};
\node (EAB) at (4,-2){$Z_3$}; 
\node (EPQ) at (6,-2){$Z_4$}; 
\node (EQ) at (8,-2){$Z_5$.}; 
\draw[->] (EA) to[out=10,in=170] node[above,sloped] {\footnotesize $\widehat\kappa_3x_B$} (EAB);
\draw[->] (EAB) to[out=190,in=-10] node[below,sloped] {\footnotesize $\widehat\kappa_4$} (EA);
\draw[->] (EAB) to[out=10,in=170] node[above,sloped] {\footnotesize $\widehat\kappa_5$} (EPQ);
\draw[->] (EPQ) to[out=190,in=-10] node[below,sloped] {\footnotesize $\widehat\kappa_6$} (EAB);
\draw[->] (EPQ) to[out=10,in=170] node[above,sloped] {\footnotesize $\widehat\kappa_7$} (EQ);
\draw[->] (EQ) to[out=190,in=-10] node[below,sloped] {\footnotesize $\widehat\kappa_8x_P$} (EPQ);
\draw[->] (E) to[out=10,in=170] node[above,sloped] {\footnotesize $\widehat\kappa_1x_A$} (EA);
\draw[->] (EA) to[out=190,in=-10] node[below,sloped] {\footnotesize $\widehat\kappa_2$} (E);
\draw[->] (EQ) to[out=230,in=310] node[below,sloped] {\footnotesize $\widehat\kappa_9$} (E);
\draw[->] (E) to[out=320,in=220] node[above,sloped] {\footnotesize $\widehat\kappa_{10}x_Q$} (EQ);
\end{tikzpicture}
}
\end{center}
The set $\Sigma$ has two cycles: The cycle with the edges with labels $\widehat\kappa_1$, $\widehat\kappa_3$, $\widehat\kappa_5$, $\widehat\kappa_7$ and $\widehat\kappa_9$  that meets all nodes clockwise, and the cycle with the edges with the rest of the labels that meets all nodes counter-clockwise, $\widehat\kappa_2$, $\widehat\kappa_4$, $\widehat\kappa_6$, $\widehat\kappa_8$ and $\widehat\kappa_{10}$.  In both cases $\Gamma(\sigma)$ is the cycle itself and the labels are respectively 
\[\widehat\kappa_1\widehat\kappa_3\widehat\kappa_5\widehat\kappa_7\widehat\kappa_9x_Ax_B\qquad \textrm{and}\qquad \widehat\kappa_2\widehat\kappa_4\widehat\kappa_6\widehat\kappa_8\widehat\kappa_{10}x_Px_Q,\]
and furthermore one has $I=\emptyset$. Blanket condition (i) is satisfied if there is a spanning tree with positive labels of the connected component with nodes $\mZ$, see Lemma \ref{lem:spanning}. The latter can be achieved in various ways:
\begin{itemize}
\item $\widehat\kappa_2>0$, $\widehat\kappa_4>0$, $\widehat\kappa_6>0$, $\widehat\kappa_8>0$,
\item $\widehat\kappa_2>0$, $\widehat\kappa_4>0$, $\widehat\kappa_6>0$, $\widehat\kappa_9>0$,
\item $\widehat\kappa_2>0$, $\widehat\kappa_4>0$, $\widehat\kappa_7>0$, $\widehat\kappa_9>0$,
\item $\widehat\kappa_2>0$, $\widehat\kappa_5>0$, $\widehat\kappa_7>0$, $\widehat\kappa_9>0$,
\item $\widehat\kappa_3>0$, $\widehat\kappa_5>0$, $\widehat\kappa_7>0$, $\widehat\kappa_9>0$,
\end{itemize}
assuming $Z_1$ to be the root, and similarly if any other node is the root. This gives $25$ different cases.
Blanket condition (ii) holds for all positive $x_A,x_B,x_P,x_Q$ if and only if $\widehat\kappa_1\widehat\kappa_3\widehat\kappa_5\widehat\kappa_7\widehat\kappa_9=0$ and $\widehat\kappa_2\widehat\kappa_4\widehat\kappa_6\widehat\kappa_8\widehat\kappa_{10}=0$, or $\widehat{\alpha}=0$. 

We have (with $z_3$ eliminated using the linear first integral) 
$$a_0(x)=\begin{pmatrix} 0 \\ \widehat\kappa_4\widehat\alpha \\ 0 \\ 0 \end{pmatrix},\quad b_0(x)=\begin{pmatrix} 0 \\ \widehat\kappa_4\widehat\alpha \\ \widehat\kappa_5\widehat\alpha \\ 0 \end{pmatrix}.$$
In particular, if $\widehat\alpha=0$ or $\widehat\kappa_4=\widehat\kappa_5=0$, then $a_0(x)=b_0(x)=0$ and a Tikhonov-Fenichel reduction exists and agrees with the QSS reduction.   By symmetry of the reactions, the same holds if  $\widehat\kappa_6=\widehat\kappa_7=0$.

\subsection{Post-translational modification systems}
\paragraph{Generalities.}
We will consider here a generalization of the Michaelis Menten system in Section \ref{sec:MM} known as post-translational modification (PTM) systems \cite{fwptm}. Mass-action kinetics is assumed throughout.

A PTM system consists of reactions of the form
$$S_i+S_j\ce{<=>[a_{i,j}^\ell][b^\ell_{i,j}]} C_\ell,\quad C_\ell \ce{<=>[c_{\ell,k}][c_{k,\ell}]} C_k,\quad S_i \ce{<=>[d_{i,j}][d_{j,i}]} S_j,$$
for $i,j,\ell,m$ varying in some index sets and  $a^\ell_{i,j},b^\ell_{i,j},c_{\ell,k},d_{i,j}\ge 0$. (Recall that a reaction rate constant is  allowed to be zero  in which case the corresponding reaction does not take place.) The species $S_i$ are known as substrates and the species $C_\ell$ as intermediates. The Michaelis-Menten system discussed earlier is one example of a PTM system (in which the enzyme also plays the role of a substrate).  PTM systems are found in abundance in biological organisms and PTM is considered a general mechanism for signal transmission \cite{fwptm}. The class of PTM systems also includes the MAPK cascade, a layered network of reactions in which a signal is filtered. These systems play pivotal roles in the modelling of cancers and have been studied extensively in the literature, experimentally as well as mathematically.

We will assume that all intermediate species  are degraded in the sense that for any $C_\ell$  there exists a sequence of reactions (with positive rate constants) such that
$$C_\ell \ce{->[c_{\ell,\ell_1}]}\quad \ldots\quad \ce{->[c_{\ell_{k-1},\ell_{k}}]} C_{\ell_k}\ce{->[b^{\ell_k}_{i,j}]} S_i+S_j.$$

Next, we will study some generic cases of non-interacting species sets in the light of Lemma \ref{lem:spanning} and Lemma \ref{lem:cycle}. Let $\kappa$ denote the vector of parameters and assume $\kappa=\widehat\kappa+\epsilon\kappa^*+\ldots$ is a curve in parameter space with $\widehat\kappa\in\R^m_{\ge 0}$.
The case where $\mathcal{Z}$ only consists of intermediate species was discussed in Remark~\ref{rk:intermediates}.

We  consider a  generalization of   the standard Michaelis-Menten reduction by enzyme and  substrate. For this, 
assume furthermore that
\begin{itemize}
\item All intermediate species are produced and degraded that is, for $C\in\mC$ (in the set of intermediate species) there is a sequence of reactions such that
$$S_i+S_j\ce{->[]}\ldots\ce{->[]}C\ce{->[]}\ldots\ce{->[]} S_k+S_\ell$$
\item $\mZ=\mC\cup\mS$, where $\mS=\{S_{K+1},\ldots,S_M\}$ (potentially after relabelling) is a subset of  the substrate species. Furthermore, assume the graph $\mG_\k$ has two components, $\mG_\k^0$ with $\mN^0=\{*\}$ and $\mG_\k^1$. Hence there is a linear first integral
$$\sum_{i\colon C_i\in\mC} z_i+\sum_{i\colon S_i\in \mS} z_i$$
relating the non-interacting species, and if there is a non-interacting species in the reactant (product) of a reaction, then there is one in the product (reactant) of the same reaction.
\end{itemize}

Assume that blanket conditions (i) and (ii) are satisfied, for example by choosing rate constants or conserved amounts such that  Lemma \ref{lem:spanning} and Lemma \ref{lem:cycle} are applicable. Note that any cycle in $\Sigma$  contains both substrates and intermediates. Indeed, a cycle of $\mG_\k$ involving only substrates (resp.\ intermediates) corresponds to a reaction path with reactions of the form $S_i\ce{->[]}S_j$ (resp.\ $C_i\ce{->[]}C_j$), hence the net production of species that are not non-interacting is zero and the cycle is not in $\Sigma$.

Since $\mG^0_\k$ contains only the node $*$, one has $b_0(x)=0$ and a Tikhonov-Fenichel reduction is exists and agrees with the QSS reduction. Furthermore, it takes the form in Proposition~\ref{zeezeroprop}, where
 $a_1(x),b_1(x)$, $A_0(x)$, $B_0(x)$ are all linear in $x$, hence the right hand side of the ODE system is a rational function $p(x)/q(x)$ in $x$, with $p(x),q(x)$ irreducible polynomials in $x$.
  It  follows from \cite{fwptm} that the monomials of $p(x)$ and $q(x)$ only depend on the reactions $S_i\ce{->[]}S_j$ and whether $S_i+S_j$ is connected by a reaction path to 
$S_k+S_\ell$ or not, and not on the chain of intermediate species connecting them nor the structure of the intermediate network as such.

\paragraph{A class of PTM systems.} As an example, we consider a modified Michaelis-Menten system with enzyme $E$, substrate $S$, product $P$ and an arbitrary number $C_1,\ldots, C_m$ of intermediate complexes (that is, $E, S,P$ are ``substrates''  in the terminology of the first part). The reactions are
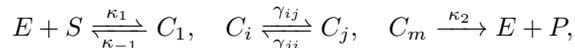
\[
E+S\ce{<=>[\k_1][\k_{-1}]} C_1, \quad C_i\ce{<=>[\gamma_{ij}][\gamma_{ji}]} C_j, \quad C_m\ce{->[\k_2]} E+P,
\]
and we assume that there is a reaction path from $C_1$ to $C_m$, and mass-action kinetics. 
We consider the non-interacting set $\mathcal{Z}=\{ E,C_1,\dots,C_m\}$ with the linear first integral $\alpha=x_E+ x_{C_1}+\dots+x_{C_m}$.
We have $b_0=0$ as the component of $*$ only contains one node.
The graph $\mG_\k$ is of the form
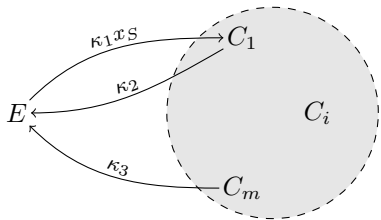
\begin{center}
\begin{tikzpicture}[inner sep=1.2pt]
\draw[dashed,fill=gray!20!white] (3.4,0) circle (40pt);
\node (E) at (0,0) {$E$};
\node (C1) at (3,1) {$C_1$};
\node (Cm) at (3,-1) {$C_m$};
\node (Ci) at (4,0) {$C_i$};
\draw[->] (E) to[out=45,in=180]  node[above,sloped]{\footnotesize $\kappa_1x_S$} (C1);
\draw[->] (C1) to[out=210,in=0]  node[above,sloped]{\footnotesize $\kappa_2$} (E);
\draw[->] (Cm) to[out=180,in=-45]  node[above,sloped]{\footnotesize $\kappa_3$} (E);
\end{tikzpicture}
\end{center}
where the dashed circle represents the edges among the intermediate species and with labels $\gamma_*$. Note that all edges of $\mG_\k$ have a label that is constant in $x_S$, except the edge $E\rightarrow C_1$ has label $\k_1x_S$. 

Blanket condition (i) is satisfied if the graph $\mG_\k$ contains at least one spanning tree. 
All cycles in $\Sigma$ must involve both $E$ and intermediates, and since the net production of either $S$ or $P$ needs to be non-zero, we are left with cycles containing the edges with labels $\k_1x_S$ and $\k_2$. 

Hence, in view of Lemma~\ref{lem:cycle}, assuming the existence of the spanning tree with positive label, $\widehat\alpha=0$ or $\widehat\k_1=0$ or $\widehat\k_3=0$ guarantee that blanket condition (ii) holds. 
By Corollary~\ref{cor:QSS}, $\widehat\alpha=0$ implies $w=0$. When $\widehat\k_1=0$, $w$ is constant as $v_1(x,\widehat\k)=0$ and $K_{2}(x,\k)$ becomes constant in $x$. 
It follows that either $\widehat\alpha=0$ or $\widehat\k_1=0$ give choices of TFPV, the Tikhonov-Fenichel reduction exists and agrees with QSS.

We now find the reduced system in $x_S$, by using the expression in \eqref{cqssredsys} and then evaluating at the TFPV. 
We have
\[
\dot{x}_S = -\k_1 x_E x_S + \k_2 x_{C_1} .
\]
We assume the graph $\mG_\k$ is strongly connected. Then, by \cite{Fel_elim}, see also \cite{saez_reduction,feliu:intermediates}, 
the solution to $\dot{x}_{C_i}=0$, $i=1,\dots,m$, together with the linear first integral is of the form
\[ x_E = \frac{\alpha (\k_2 \theta_1 + \k_3 \theta_m)}{\k_1\delta x_S + \k_2 \theta_1 + \k_3 \theta_m},\qquad x_{C_i} = \frac{\alpha \k_1 \theta_i x_S}{\k_1\delta x_S + \k_2 \theta_1 + \k_3 \theta_m},\]
where 
 $\theta_i $ is the sum of the labels of all spanning trees of the subgraph of $\mG_\k$ delimited by the dashed circle in the figure above rooted at $C_i$, and depends only on $\gamma_*$; and  $\delta$ is the sum $\theta_1+\dots+\theta_m$.
Then $\k_2 \theta_1 + \k_3 \theta_m$ is the sum of the labels of all spanning trees of $\mG_\k$ rooted at $E$ and $\k_1\theta_i x_S$  is the sum of the labels of all spanning trees of $\mG_\k$ rooted at $C_i$. The denominator is the sum of all possible spanning trees of $\mG_\k$. 
Plugging these expressions into $\dot{x}_S$, we obtain
\[ 
\dot{x}_S = -\k_1  x_S  \frac{\alpha (\k_2 \theta_1 + \k_3 \theta_m) }{\k_1\delta x_S + \k_2 \theta_1 + \k_3 \theta_m} + \k_2 \frac{\alpha \k_1 \theta_1 x_S}{\k_1\delta x_S + \k_2 \theta_1 + \k_3 \theta_m} = \frac{ - \k_1 \k_3 \alpha  \theta_m x_S}{\k_1\delta x_S + \k_2 \theta_1 + \k_3 \theta_m} .
\] 
This is the QSS reduction of the system, which agrees with the Tikhonov-Fenichel reduction if  either $\widehat\alpha=0$ or $\widehat\k_1=0$, provided $\mG_{\widehat{\k}}$ has a rooted spanning tree with positive label. Remarkably, the basic form of the reduced equation (the right hand side being a quotient of two degree one polynomials) does not depend on the number of intermediates nor on specifics of their interactions, and is identical with the form of the standard Michaelis-Menten equation.

We now look at the specific cases. For $\widehat\alpha=0$, we consider the curve $\epsilon \alpha^*$ in parameter space, which gives in slow time 
\[ 
\frac{d x_S}{d\tau} =  \frac{ - \widehat\k_1 \widehat\k_3 \alpha^*  \widehat\theta_m x_S}{\widehat\k_1\widehat\delta x_S + \widehat\k_2 \widehat\theta_1 +\widehat \k_3\widehat \theta_m} .
\] 

For $\widehat\k_1=0$, we consider $\epsilon\k_1^*$, which gives in slow time
\[ 
\frac{d x_S}{d\tau}= \frac{ -\k_1^*  \widehat\k_3 \widehat\alpha  \widehat\theta_m x_S}{\epsilon\k_1^* \widehat\delta x_S + \widehat\k_2\widehat\theta_1 + \widehat\k_3 \widehat\theta_m} = \frac{ -\k_1^*  \widehat\k_3 \widehat\alpha  \widehat\theta_m }{\widehat\k_2\widehat\theta_1 + \widehat\k_3 \widehat\theta_m}x_S+\epsilon(\ldots).
\]

\section*{Appendix: A brief outline of singular perturbation reduction}
Here we give a brief  informal outline on singular perturbation reduction according to Tikhonov\cite{tikh} and Fenichel \cite{fenichel}. For more details see the monograph by Verhulst \cite{verhulst}, Chapter 8, and \cite{gw2} for the coordinate-independent version. All functions and vector fields in the following are assumed to be sufficiently differentiable.
\begin{enumerate}
\item Consider a system with small parameter $\varepsilon$ in {\em standard form}
\[
\begin{array}{rccl}
 \dot x_1 &=  f_1(x_1,\,x_2) + \varepsilon \;(\dotsc),\quad &\quad &x_1 \in D\subseteq \mathbb R^r,\\
 \dot x_2 &=  \varepsilon f_2(x_1,\,x_2) + \varepsilon^2 \;(\dotsc),\quad &\quad &x_2 \in G \subseteq\mathbb R^s.
\end{array}
\]
Rewritten in {\em slow time} $\tau=\varepsilon t$ one obtains
\[
 \varepsilon  x_1^\prime= f_1(x_1,\,x_2) + \cdots,\quad
  x_2^\prime =f_2(x_1,\,x_2) + \cdots.
\]
Given that
\begin{itemize}
\item there is a non-empty {\em critical manifold}
\[
\widetilde Z :=\left\{ (y_1,\,y_2)^T\in D\times G ;\, f_1(y_1,\,y_2) =0\right\};
\]
\item there exists $\nu>0$ such that all eigenvalues of $ D_1f_1(y_1,\,y_2)$, $(y_1,y_2)\in\widetilde Z$ have real part $\leq -\nu$,
\end{itemize}
then by 
{\bf Tikhonov's  Theorem} there exist $T>0$ and a neighborhood of $\widetilde Z$ in which, as $\varepsilon \to 0$, all solutions converge uniformly to solutions of
\[
 x_2^\prime = f_2(x_1,\,x_2),\quad f_1(x_1,\,x_2)=0 \quad \text{on  }\left[ t_0,\,T\right] 
\]
with $t_0>0$ arbitrary.
\item
More generally, a system may be put into standard form (and then admit a singular perturbation reduction) by a coordinate transformation. Thus we start with a parameter dependent equation
\[
 \dot x = h^{(0)}(x) + \varepsilon h^{(1)}(x) +\varepsilon^2 \dotso
\]
and assume that $Z:=\{x;\,h^{(0)}(x)=0\}$ has dimension $s>0$. This system
admits a coordinate transformation into standard form and subsequent Tikhonov-Fenichel reduction near every point of $Z$ if and only if
\begin{enumerate}[(i)]
\item  ${\rm rank\ } Dh^{(0)}(x)=r:=n-s$ for all $x\in Z$,
\item for each $x\in Z$ there exists a direct sum decomposition
$\mathbb R^n = {\rm Ker\ } Dh^{(0)}(x) \oplus  {\rm Im\ } Dh^{(0)}(x)$,
\item for each $x\in Z$ the non-zero eigenvalues of $Dh^{(0)}(x)$ have real parts smaller than $ -\nu<0$. 
\end{enumerate}
\item
The remaining problem is that an explicit computation of the coordinate transformation is generally impossible. This can be circumvented by the following coordinate-free reduction procedure, which we state for the system
\[
 x^\prime =  \varepsilon^{-1}h^{(0)}(x) + h^{(1)}(x) +\dotso
\]
in slow time. We assume that $Z\subseteq\mathcal{V}(h^{(0)})$, the vanishing set of $h^{(0)}$, satisfies conditions (i), (ii) und (iii), and let $a\in Z$.

\smallskip
{\bf Decomposition:} There is an open neighborhood $U_a$ 
of \(a\) such that
\[
h^{(0)}(x) = P(x) \mu(x), 
\]
with \(\mu(x)$ having values in $\mathbb R^{r}$, $P(x)$ having values in $ \mathbb R^{n\times r}$, \({\rm rank}\ P(a) = r\), ${\rm rank}\ D\mu(a)=r$, and (w.l.o.g.)
$\mathcal{V}(h^{(0)}) \cap U_a = \mathcal{V}(\mu) \cap U_a=Z$. (This is a consequence of the implicit function theorem for the differentiable case. When $h^{(0)}$ is rational then $P$ and $\mu$ can be chosen rational, and $U_a$ is Zariski-open.

\smallskip
{\bf  Reduction:} The system
\[
 x' =  \left [I_n - P(x) A(x)^{-1} D\mu(x)\right] h^{(1)}(x), \quad \text{with  }A(x):= D\mu(x) P(x)
 \]
is defined on $U_a$ and admits $Z$ as invariant set. The restriction to $Z$ corresponds to the reduction from Tikhonov's theorem.
\end{enumerate}
\bigskip
\bigskip

\noindent{\bf Acknowledgements.} CL acknowledges support by the DFG Research Training Group GRK 1632 ``Experimental and Constructive Algebra''.   SW acknowledges support  by the bilateral project ANR-17-CE40-0036 and DFG-391322026 SYMBIONT. EF and CW acknowledge support from the Independent Research Fund of Denmark. Moreover EF, SW and CW thank the Erwin-Schr\"odinger-Institute (Vienna) for the opportunity to participate, in October 2018, in the workshop ``Advances in Chemical Reaction Network Theory'', during which a substantial part of the present paper was written. 


\end{document}